\documentclass{article}
\usepackage{graphicx}
\usepackage{amsthm,amsmath,amssymb}
\newcommand{\conv}{\mbox{conv}}
\newtheorem{lem}{Lemma}
\newtheorem{thm}[lem]{Theorem}
\newtheorem{exa}[lem]{Example}
\newtheorem{cor}[lem]{Corollary}
\newtheorem{con}[lem]{Conjecture}
\newtheorem{prop}[lem]{Proposition}
\newtheorem{df}[lem]{Definition}
\begin{document}
\title{The Weak-Map Order and Polytopal Decompositions of Matroid Base Polytopes}
\author{Kenji Kashiwabara\\
Department of Systems Science, University of Tokyo\\
  3-8-1 Komaba, Meguro, Tokyo, 153-8902, Japan
 }
\maketitle
\begin{abstract}
The weak-map order on the matroid base polytopes is the partial order defined by inclusion. 
Lucas proved that the base polytope of no binary matroid includes 
the base polytope of a connected matroid.
A matroid base polytope is said to be decomposable when it has a polytopal decomposition which consists of at least two matroid base polytopes.
 We shed light on the relation between the decomposability and the weak-map order of matroid base polytopes.
We classify matroids into five types with respect to the weak-map order and decomposability.
We give an example of a matroid in each class. Moreover, we give a counterexample to a conjecture proposed by Lucas, which says that, when one matroid base polytope covers another matroid base polytope with respect to inclusion, the latter matroid base polytope should be a facet of the former matroid base polytope. 
   
\end{abstract}


\section{Introduction}\label{sec:intro}

The set ${\cal B}(M)$ of the bases of a matroid $M$ is called a {\it matroid base system}.
The base system of $M$ can be identified with the base polytope $B(M)$.
The weak-map order that we consider is a partial ordering defined on the matroid base systems of the fixed rank and the fixed finite ground set $E$.
The weak-map order is defined according to the inclusion relation among their bases. 
For matroids $M_1$ and $M_2$, $M_1 \succeq M_2$ in the {\it weak-map order}  is defined by ${\cal B}(M_2)\subseteq {\cal B}(M_1)$.
We consider the polytopal decomposition of the base polytope of a matroid. Roughly
speaking, a matroid base polytope is said to be {\it decomposable} when it has a
polytopal decomposition which consists of at least two matroid base polytopes of the same dimension. That is, $B(M)=\bigcup_i B(M_i)$ where $M_i$ are matroids and any intersection $B(M_i)\cap B(M_j)$ is a facet of both $B(M_i)$ and $B(M_j)$.
When a matroid base polytope has a decomposition which has exactly two matroid base polytopes of the same dimension, such a polytope is said to be {\it 2-decomposable} (or have a hyperplane-split in the literature.)

The weak-map order is related to the polytopal decomposition of a base polytope.
By definition, the base polytopes obtained from a polytopal decomposition of the base polytope $B(M)$
 are smaller than the base polytope $B(M)$ with respect to weak-map order.
Note that the base polytope of a connected matroid is of dimension $|E|-1$.
Therefore a matroid base polytope which is minimal in all the connected matroids on $E$ with respect to weak-map order is not decomposable.

The connected matroids are classified into the following five types. (a) Binary matroids. (b) Non-binary but minimal matroids in the connected matroids with respect to inclusion. (c) Non-minimal but indecomposable matroids. (d) Non-2-decomposable but decomposable matroids. (e) 2-decomposable matroids.
An example of type (b) of rank 4 can be found in Lucas \cite{Lucas75}. 
We give an example of type (b) of rank 3 in Example \ref{exa:minimal}.
We give an example of type (c) in Example \ref{exa:nonminimal}.
We give an example of type (d) in Example \ref{exa:typed}.

Lucas \cite{Lucas75} conjectured that, when one matroid base polytope covers another matroid base polytope in the poset of the matroid base polytopes with
respect to inclusion, the latter matroid base polytope might be a facet 
of the former matroid base polytope. We give a counterexample to this conjecture in Example \ref{exa:lucascon} and Theorem \ref{thm:lucascon}.

The problem of polytopal decomposition arises from that of M-convex functions.
 An {\it integral base polytope} is the base polytope of some integral submodular function.
A matroid base polytope is the integral base polytope for the rank function of a matroid.
  An integral base set is the set of the lattice points of an integral base polytope.
 An M-convex function is a function which satisfies some kind of exchange axiom.
 The domain of an M-convex function is an integral base set(Murota \cite{Murota96a,Murota96b}).
 A function defined on an integral base set induces a coherent polytopal decomposition of the integral base
polytope that is the convex hull of the integral base set.
 Such a coherent polytopal decomposition consists of integral base polytopes
when the function is an M-convex function.

Kashiwabara \cite{Kashiwabara01} investigated a polytopal
decomposition which consists of the integral base polytopes of an integral submodular function,
 called an integral-base decomposition, and showed that
 an integral base polytope can be divided into two integral base polytopes if and only if
 there exists a hyperplane such that the cross-section of the polytope by it is an integral base polytope.
It was shown that any integral base
polytope which is not a matroid base polytope by any translation is 2-decomposable.
 Therefore we have only to consider a polytopal decomposition of a matroid base polytope when we consider its decomposability.

This paper is organized as follows.

In Section \ref{sec:polytope}, we consider the representations of the matroid independence 
polytopes and matroid base polytopes by linear inequalities which have 01-normal vectors. 
A matroid base polytope can be determined by the family of sets which satisfy linear inequalities. We call such a family a matroid base system. 
A flat with its rank behaves like a linear inequality for a matroid base system.
We often identify a matroid base polytope with a matroid base system.
We investigate the combinatorial structures of matroid base systems.

In Section \ref{sec:weak}, we consider the weak-map order and polytopal decompositions of matroid base systems. 
In Section \ref{sec:2decomp}, we consider the 2-decomposability of a matroid base system.

In Section \ref{sec:rank3}, we consider the decomposability of a matroid base system of rank 3 in terms of graphs.
Matroid base systems of rank 3 will be used as important examples in Section \ref{sec:further}.

Section \ref{sec:further} is the main part of this paper. It consists of two subsections.
In Section \ref{subsec:classification}, by using Theorem \ref{thm:lucas}, we classify matroid base systems into five types.
In Section \ref{subsec:counter}, we give a counterexample to the conjecture proposed by Lucas.

\section{Representations of matroid systems}\label{sec:polytope}

Let $E$ be a finite ground set with $|E|\geq 2$ throughout this paper.

\subsection{Representation of independence systems}

In this subsection, we consider independence systems which may not be a matroid independence system. Moreover we introduce notation to represent a set system by 
linear inequalities and hyperplanes with 01-coefficients. This notation is different from standard one. However, we believe that this notation is useful to describe matroid systems from a polytopal viewpoint.

We prepare notation for hyperplanes and linear inequalities whose
coefficients are in $\{0,1\}$. 
A set system, which is a family of sets, can be represented by linear inequalities and hyperplanes.
For a nonnegative integer $a$ and $A\subseteq E$, we call $(A,a)_\leq =\{I\subseteq E |\ |I \cap A| \leq a\}$ a (closed) {\it linear inequality}. 
A family $(A,a)_\leq$ is identified with the 01-points $\{p\in
\{0,1\}^E|\langle \chi_A ,p\rangle \leq a\}$ where $\chi_A$ is the incidence vector of $A$ and $\langle \chi_A, p\rangle$ is the inner product of $\chi_A$ and $p$.
We write $(A,a)_>=\{I\subseteq E |\ |I \cap A| >a\}$ and so on.
Denote $(A,a)_==\{I\subseteq E |\ |I \cap A| = a\}$ where
$a$ is an integer with $0 \leq a \leq |A|$ and $\emptyset\neq A \subseteq E$. A
family $(A,a)_=$ is identified with the 01-points $\{p\in
\{0,1\}^E|\langle \chi_A ,p\rangle = a\}$ on hyperplane $\langle\chi_A,p\rangle =a$. Therefore we call $(A,a)_=$ a {\it hyperplane}.

\begin{lem}\label{lem:facetinc}
$(A_2,a_2)_\leq\subseteq (A_1,a_1)_\leq$ if and only if $|A_1-A_2|\leq a_1 - a_2$.
\end{lem}

\begin{proof}
 Assume $|A_1-A_2|\leq a_1 - a_2$.
 Let $D\notin (A_1,a_1)_\leq$, that is, $D\in (A_1,a_1)_>$.
 Then $|D\cap A_1|>a_1$.
 Therefore $a_1< |D\cap A_1|\leq |D\cap A_2|+|A_1-A_2|\leq |D\cap A_2|+a_1-a_2.$
 Therefore we have $|D\cap A_2| > a_2$ and $D\notin (A_2,a_2)_\leq$.

\vspace{3mm}

Conversely, assume $|A_1-A_2|>a_1 - a_2$.

In the case $|A_1-A_2|>a_1+1$, we can take $D\in (A_1,a_1)_>$ so
that $D\subseteq A_1-A_2$. We have $D\in (A_2,a_2)_\leq$ because of $D\cap A_2=\emptyset$.

In the case $|A_1-A_2|\leq a_1+1$, we can take $D\in (A_1,a_1)_>$ so
that $|D|= a_1+1$ and $A_1-A_2\subseteq D$.
At that time, we have $|D\cap A_2|+|D\cap (A_1-A_2)|\leq |D|=a_1+1$ since $D\cap A_2$ and $D\cap (A_1-A_2)$ are disjoint.
 By assumption $|A_1-A_2|>a_1 - a_2$, we have $|D\cap A_2|\leq a_1+1-|D\cap (A_1-A_2)| =
a_1+1-|A_1-A_2|<a_1+1-(a_1-a_2)=a_2+1.$
 Therefore $|D\cap A_2|\leq a_2$, that is, $D\in (A_2,a_2)_\leq$.
\end{proof}

A non-empty family ${\cal I}$ on $E$ is called an {\it independence system}
if $I_1\in {\cal I}$ and $I_2\subseteq I_1$ imply $I_2\in {\cal I}$.
An independence system ${\cal I}$ on $E$ is said to be {\it represented} by a subset ${\cal F}$ of linear inequalities if 
$$I\in {\cal I}\Leftrightarrow I\in (A,a)_\leq \mbox{ for all } (A,a)_\leq\in {\cal F}.$$ 
${\cal F}$ is called a {\it representation} of ${\cal I}$. It is known that every
independence system has such a representation.
Many familiar notions, for example, rank
functions, flats, and bases, in matroid theory are also defined for independence systems.

A {\it circuit} is a minimal dependent set.
For an independence system ${\cal I}$,
 the {\it rank function} $r$ is defined by $r(A)=\max\{|I|\ | I\in {\cal I},I\subseteq A\}=\max\{|I\cap
A|\ |I\in {\cal I}\}$.
 $r(E)$ is called the rank of the independence system ${\cal I}$.
$A\subseteq E$ is called a {\it flat} for an
independence system ${\cal I}$ if $A=B$ holds whenever $A\subseteq B
\subseteq E$ and $r(A)=r(B)$.

Every independence system ${\cal I}$ is expressed as 
$${\cal I}=\{I\subseteq E|\ |A\cap I|\leq
r(A) \mbox{ for all }A \subseteq E\}=\bigcap_{A\subseteq E}(A,r(A))_\leq=\bigcap_{A:\mbox{flat}}(A,r(A))_\leq.$$
 where $r$ is the rank function of ${\cal I}$.


For an independence system ${\cal I}$, we call $(A,a)_\leq $ {\it
  valid} for ${\cal I}$ when ${\cal I}\subseteq (A,a)_\leq$.
 For a representation ${\cal F}$ of ${\cal I}$, $(A,a)_\leq \in {\cal F}$ is valid.
Note that, for an independence system ${\cal I}$, $(A,a)_\leq$ is valid if and only if $r(A)\leq a$. 
Especially, $(A,r(A))_\leq$ is valid.

\subsection{Representations of matroid independence systems}

 When an independence system satisfies the matroid augmentation axiom (see, e.g.  \cite{Oxley92}),
the independence system is called a {\it matroid independence system}.

We consider a condition for an independence system to be a matroid independence system. 
The next theorem is due to Conforti and Laurent \cite{Conforti88}.

A pair of linear inequalities $(A_1,a_1)_\leq$ and $(A_2,a_2)_\leq$ is said to be 
{\it intersecting} if $A_1\cap A_2\neq \emptyset,A_1-A_2\neq \emptyset$ and $A_2-A_1\neq \emptyset$.

\begin{thm}\cite{Conforti88}\label{thm:facetsub}
  Let ${\cal I}$ be an independence system represented by a set ${\cal
    F}$ of linear inequalities.
 ${\cal I}$ is a matroid independence system if and only if,
 for any intersecting inequalities $(A_1,a_1)_\leq\in {\cal F}$
  and $(A_2,a_2)_\leq\in {\cal F}$, $r(A_1)+r(A_2) \geq r(A_1\cap A_2)+r(A_1 \cup A_2)$ holds.
\end{thm}


It is known that, for any matroid, the intersection of any two flats is a flat again. 
Therefore the set of flats of a matroid is a closure system.
The closure $\mbox{cl}(A)$ of $A\subseteq E$ is the minimum flat including $A$.

The  minimum representation is a representation which has no redundant inequalities. 
Greene \cite{Greene91} showed the following theorem. 

\begin{thm}(Greene \cite{Greene91})\label{thm:greene}
Every matroid independence system has
the unique minimum representation defined by flats with respect to set
inclusion among all the representations defined by flats.
Moreover, the minimum representation is $\{(A,r(A))_\leq|A \mbox { is the closure of a circuit.}\}$. 
\end{thm}


\vspace{5mm}

We consider the independence polyhedron
$P(M)=\{p\in {\bf R}^E|\langle \chi_A , p\rangle \leq r(A) \mbox{ for all
  }A\subseteq E\}$ of a matroid $M$.
 We can also represent a matroid as a family of sets.
By identifying a set $A\subseteq E$ with its incidence vector
$\chi_A$, the family ${\cal I}(M)$ is identified with $\{p\in
\{0,1\}^E|\langle \chi_A, p\rangle \leq r(A) \mbox { for all }A \subseteq E\}$.
Therefore the independent sets of a matroid can be considered
as the 01-points in the polyhedron expressed by linear inequalities
which have 01-normal vectors and satisfy submodularity. 

For valid inequalities $\{(A_i,r(A_i))_\leq\}_i$ to a matroid
independence system ${\cal I}(M)$, ${\cal I}(M)\cap (\bigcap_i
(A_i,r(A_i))_=)$ is called a {\it face} of the independence system of the matroid $M$. 
A face is {\it proper} if it is not ${\cal I}(M)$.
Any face of matroid independence system ${\cal I}(M)$
corresponds to some face of the matroid independence polyhedron $P(M)$.
A face of a matroid independence system is called a {\it facet} if it
is maximal in all the proper faces with respect to inclusion.
When ${\cal I}(M) \cap (A,r(A))_=$ is a facet, $(A,r(A))_\leq$ is called {\it a facet-defining inequality} of ${\cal I}(M)$.

\begin{lem}\label{lem:indflat}
For a loopless matroid $M$, if $(A,r(A))_\leq$ is a facet-defining inequality, $A$ is a flat.
\end{lem}

\begin{proof}
Suppose that $A$ is not a flat. $\mbox{cl}(A)$ denotes the closure of $A$. Then $r(\mbox{cl}(A))=r(A)$. 
We can take $x\in \mbox{cl}(A)-A$ since $A$ is not a flat. 
Since the restriction of loopless matroid $M$ to $\mbox{cl}(A)$ is also a loopless matroid, there exists $B\in {\cal I}(M)\cap (\mbox{cl}(A),r(A))_=$ with $x\in B\subseteq \mbox{cl}(A)$. 
Then $B\notin {\cal I}(M) \cap (A,r(A))_=$ since $|B\cap A|<|B\cap (A\cup x)|\leq |B\cap \mbox{cl}(A)| =|B|=r(A)$. 
Since, for any $X\in {\cal I}(M)$, $|X\cap A|= r(A)$ implies $|X\cap \mbox{cl}(A)|=r(A)$, we have ${\cal I}(M) \cap (A,r(A))_= \subseteq {\cal I}(M) \cap (\mbox{cl}(A),r(\mbox{cl}(A)))_=$. Hence we have ${\cal I}(M) \cap (A,r(A))_= \subsetneq {\cal I}(M) \cap (\mbox{cl}(A),r(\mbox{cl}(A)))_=$. Therefore $(A,r(A))_\leq$ is not a facet-defining inequality.
\end{proof}

For a loopless matroid, we call a flat $A$ a {\it facet-defining flat} of ${\cal I}(M)$ if $(A,r(A))_\leq$ is a facet-defining inequality of ${\cal I}(M)$.

The next theorem is due to Edmonds.

\begin{thm}(Edmonds \label{thm:bifacet}\cite{Edmonds71})
  For a loopless matroid independence system ${\cal I}(M)$, ${\cal
    I}(M)\cap (A,r(A))_=$ is a facet of ${\cal I}(M)$ if and only if
  $A$ is a flat and the restriction $M|A$ of $M$ to $A$ is connected.
\end{thm}

We discuss the relation between Theorem \ref{thm:greene} and Theorem \ref{thm:bifacet}.
When $A$ is the closure of some circuit, $A$ is a flat and $M|A$ is connected since any circuit is included in some connected component.
Therefore we have ${\cal I}(M)=\bigcap_{\{\mbox{$A$ is a facet-defining flat}\}}(A,r(A))_\leq.$
However, it may not be the minimum representation by flats.

\begin{exa}
Consider the matroid $M$ on $E=\{a,b,c,d,e,f\}$ defined by 
$${\cal I}(M)=(E,4)_\leq\cap (\{a,b,c,d\},3)_\leq \cap (\{c,d,e,f\},3)_\leq \cap (\{e,f,a,b\},3)_\leq.$$
$E$ is the closure of no circuit.
However, $E$ is a flat and $M|E$ is a connected matroid.
Therefore $(E,4)_\leq$ is a facet-defining inequality of ${\cal I}(M)$ but does not belong to the minimum representation by flats.
In fact, ${\cal I}(M)=(\{a,b,c,d\},3)_\leq \cap (\{c,d,e,f\},3)_\leq \cap (\{e,f,a,b\},3)_\leq.$
\end{exa}

\subsection{Representations of matroid base systems}\label{subsec:base}

For a matroid, a {\it base} is a maximal independent set with respect to inclusion. 
For a matroid $M$, the family of its bases is denoted by ${\cal B}(M)$ and called the {\it matroid base system} of $M$. We call a family of sets a {\it matroid base system} when it is the matroid base system of some matroid. Since any base of a matroid has the same cardinality $r(E)$, we have ${\cal B}(M)={\cal I}(M)\cap (E,r(E))_=$.

We can identify the base system ${\cal B}(M)$ of a matroid $M$ with the base polytope 
$$B(M)=\{p\in {\bf R}^E|\langle p,\chi_A \rangle \leq r(A) \mbox{ for
  all } A\subseteq E, \langle p,\chi_E\rangle = r(E)\}$$
since the bases of a matroid correspond to the extreme points of the
base polytope so that $B(M)=\conv\{\chi_B|B\in {\cal B}(M)\}$.

We try to describe the combinatorial structures of a matroid base system in terms of a family of sets. 
For a set of valid inequalities $\{(A_i,r(A_i))_\leq\}_i$ to a matroid $M$, ${\cal B}(M)\cap
(\bigcap_i (A_i,a_i)_=)$ is called a {\it face} of the matroid base system ${\cal B}(M)$.
A face is said to be {\it proper} when it is not equal to ${\cal B}(M)$.

\begin{prop}
The faces of the form ${\cal B}(M)\cap (\bigcap_i (A_i,r(A_i))_=)$ correspond to the faces
of matroid base polytope $B(M)$ bijectively.
\end{prop}

\begin{proof}
Assume that a face of $B(M)$ is given.
Since every facet of $B(M)$ has a normal vector $\chi_A$ with 01-coefficients, it can be written as ${\cal B}(M)\cap (A,r(A))_=$. Since every face of $B(M)$
can be written as the intersection of some facets of $B(M)$, the face corresponds to some ${\cal B}(M)\cap (\bigcap_i (A_i,r(A_i))_=)$.

Conversely, assume that a face of the  form ${\cal B}(M)\cap (\bigcap_i (A_i,r(A_i))_=)$ is given. By considering the convex hull of the face, we have the corresponding face of $B(M)$.
\end{proof}

It is known that every face of a matroid base system is also a matroid base system.

A face of a matroid base system is called a {\it facet} if it is
maximal in all the proper faces with respect to set inclusion.
A facet is expressed as ${\cal B}(M)\cap (A,r(A))_=$. 
We call an inequality $(A,r(A))_\leq$ a {\it facet-defining inequality} when ${\cal B}(M)\cap (A,r(A))_=$ is a facet.
For a matroid, $A$ is a flat when $(A,r(A))_\leq$ is a facet-defining inequality by Lemma \ref{lem:indflat} and Theorem \ref{thm:facetbase}.
The facets of a matroid base system are in one-to-one correspondence with the facets of the base polytope $B(M)$.

A matroid is not {\it connected} if it is expressed as ${\cal B}(M)={\cal B}(M_1)\oplus {\cal B}(M_2)=\{B_1\cup B_2|B_1\in {\cal B}(M_1), B_2\in {\cal B}(M_2)\}$ such that the non-empty underlying sets of $M_1$ and $M_2$ are disjoint.
The number of the connected components of $M$ is closely related to the 
dimension of the base polytope(e.g. \cite{Bixby85}):
$$\mbox{(the dimension of the base polytope)} = |E| - \mbox{(the number of its connected components).}$$

A face of a base polytope is a facet when the dimension of the face is less than that of the base polytope by 1.
Therefore we know that a face of a matroid base polytope is a facet if and only if
the number of the connected components of the matroid corresponding to the
face is more than that of connected components of the whole matroid by 1.
 Especially, the number of its connected components of
the matroid corresponding to a facet of the base polytope of a connected matroid is $2$.
We define the dimension of a base system as the dimension of the corresponding base polytope, that is equal to ($|E| -$ the number of connected components). Note that the dimension of a base system is irrelevant to the rank of the matroid.









For a connected matroid $M$, a facet-defining inequality $(A,r(A))_\leq$ of ${\cal B}(M)$ is {\it
  non-trivial} if $(E,r(E))_=\cap (A,r(A))_>$ and $(E,r(E))_=\cap (A,r(A))_<$ are not empty.
Otherwise a facet-defining inequality is said to be {\it trivial}.
The next lemma follows from the definition.

\begin{lem}\label{lem:note}
$(E,r(E))_=\cap (A,r(A))_\leq=(E,r(E))_=\cap(A^c,r(E)-r(A))_\geq$.
\end{lem}

\vspace{5mm}

\begin{exa}
  Let $E=\{a,b,c,d,e\}$.
 Consider the independence system ${\cal I}(M)=(E,3)_\leq\cap (\{a,b,c\},2)_\leq\cap (\{c,d,e\},2)_\leq$,
  illustrated in the left of Figure \ref{fig:m2}. A flat $(A,r(A))_\leq$ in the minimal representation of the independence system is represented by a closed curve that encircles the elements of $A$ in the figure. By Theorem \ref{thm:facetsub}, it is a matroid independence system.
We want to specify all the facets of matroid base system 
$${\cal B}(M)={\cal I}(M)\cap (E,3)_==(E,3)_=\cap (\{a,b,c\},2)_\leq\cap (\{c,d,e\},2)_\leq.$$ 


This matroid has two non-trivial facets, ${\cal B}(M)\cap
(\{a,b,c\},2)_=$ and ${\cal B}(M)\cap(\{c,d,e\},2)_=$.
  Let $M_1$ be the matroid which corresponds to a facet ${\cal B}(M)\cap
(\{a,b,c\},2)_=={\cal B}(M_1)$ of ${\cal B}(M).$
 Since $(E,3)_=\cap (\{a,b,c\},2)_\geq=(E,3)_=\cap (\{d,e\},1)_\leq$ by Lemma \ref{lem:note}, 
${\cal B}(M_1)=(E,3)_=\cap (\{a,b,c\},2)_\leq \cap
(\{d,e\},1)_\leq=(\{a,b,c\},2)_= \cap (\{d,e\},1)_=$,
 illustrated in the center of Figure \ref{fig:m2}. 
This is a non-connected matroid, which has two connected components $\{a,b,c\}$ and $\{d,e\}$.
 Therefore ${\cal B}(M_1)$ is really a facet of ${\cal B}(M)$.

  \begin{figure}[ht]
    \begin{center}
      \leavevmode
\includegraphics[width=35mm]{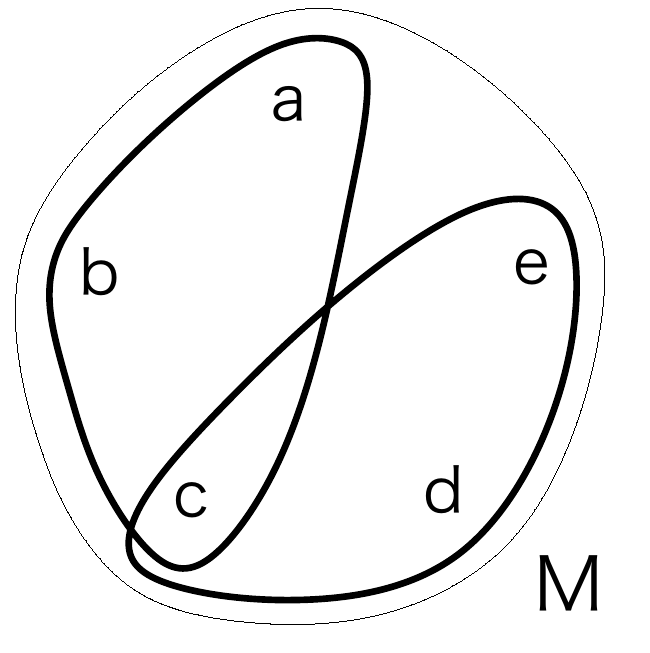}\quad
\includegraphics[width=35mm]{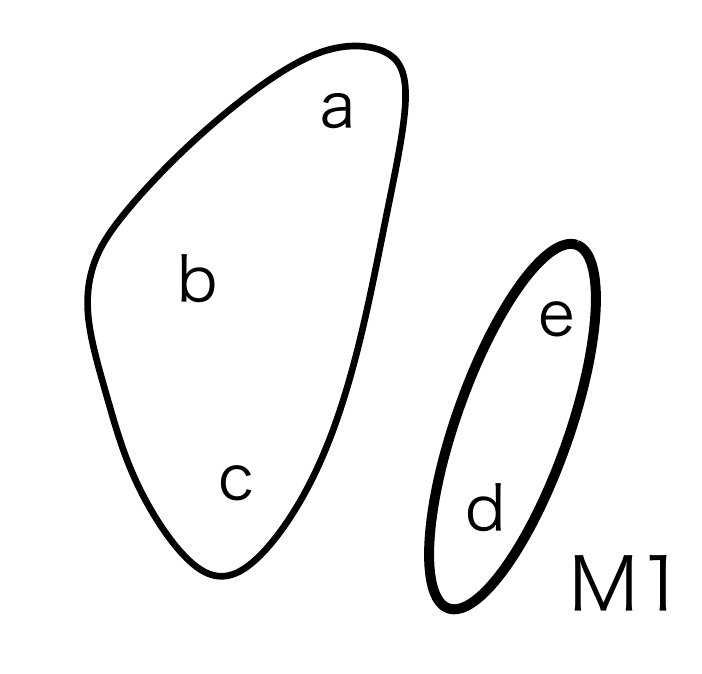}\quad\includegraphics[width=33mm]{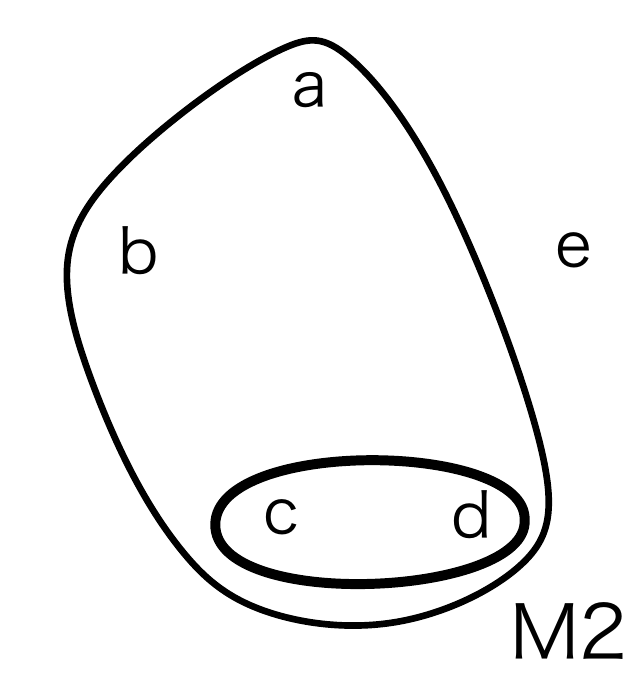}      
    \end{center}
    \caption{${\cal B}(M_1)$ and ${\cal B}(M_2)$ are facets of a matroid base system ${\cal B}(M)$.}\label{fig:m2}
  \end{figure}

We consider the face ${\cal B}(M_2)$ defined by ${\cal B}(M)\cap (\{e\},1)_=$.
Since $(\{c,d\},1)_\leq\subseteq (\{c,d,e\},2)_\leq$,
  \begin{eqnarray*}
{\cal B}(M_2)&=&{\cal B}(M)\cap (\{e\},1)_==(E,3)_=\cap (\{a,b,c\},2)_\leq\cap (\{c,d,e\},2)_\leq\cap (\{e\},1)_=\\
&=& (E,3)_=\cap (\{a,b,c,d\},2)_=\cap (\{c,d,e\},2)_\leq\cap (\{e\},1)_=\\
&=&(E,3)_=\cap (\{a,b,c,d\},2)_=\cap(\{c,d,e\},2)_\leq\cap (\{c,d\},1)_\leq\cap (\{e\},1)_=\\
&=&(\{a,b,c,d\},2)_= \cap (\{c,d\},1)_\leq \cap (\{e\},1)_=.
 \end{eqnarray*}
 Since ${\cal B}(M_2)$ has two connected components $\{a,b,c,d\}$ and $\{e\}$,
 ${\cal B}(M_2)=(\{e\},1)_=\cap {\cal B}(M)$ is a trivial facet of ${\cal B}(M)$.

However, $(\{c\},1)_=\cap {\cal B}(M_1)=(\{a,b\},1)_=\cap (\{c\},1)_=\cap
(\{d,e\},1)_=$ is not a facet of ${\cal B}(M)$ since it has three
connected components, $\{a,b\},\{c\}$ and $\{d,e\}$ as a matroid.

  $(\{a,b,d,e\},3)_=\cap {\cal B}(M)=(E,3)_=\cap
  (\{a,b,d,e\},3)_\leq \cap(\{c\},0)_\leq=(\{a,b,d,e\},3)_=\cap
  (\{c\},0)_=$. Therefore $(\{a,b,d,e\},3)_\leq$ is a trivial facet-defining
  inequality of ${\cal B}(M)$.

In summary, the base polytope $B(M)$ of $M$ has the following facets.
\begin{eqnarray*}
B(M)=\left\{p\in {\bf R}^E| \begin{array}[h]{l}
 p(\{a,b,d,e\}) \leq 3\\
 p(\{a\})\leq 1,p(\{b\}) \leq 1,p(\{d\})\leq 1,p(\{e\})\leq 1\\
 p(\{a,b,c\}) \leq 2, p(\{c,d,e\}) \leq 2,p(E)=3
  \end{array}\right\}
\end{eqnarray*}
where $p(A)=\langle p, \chi_A \rangle.$

\end{exa}

\begin{prop}
For a matroid $M$, ${\cal I}(M)\subseteq (A,a)_\leq$ if and only if ${\cal B}(M)\subseteq (A,a)_\leq$.
\end{prop}

\begin{proof}
Only-If-part follows from ${\cal B}(M)\subseteq {\cal I}(M)$. 

We show If-part. Assume that ${\cal I}(M)\subseteq (A,a)_\leq$ does not hold. Then $|I\cap A|>a$ for some $I\in {\cal I}(M)$. 
Then there exists $B\in {\cal B}(M)$ such that $I\subseteq B$.
We have $|B\cap A|>a$. Therefore $B\in {\cal B}(M)$ and $B\notin (A,a)_\leq$.
\end{proof}

Next, we consider the relation between base systems and minor operations on matroids. 

Consider contraction $M/A$ and deletion $M\backslash A$ by $A\subseteq E$ for a matroid $M$.
$M|A$ denotes the restriction of $M$ by $A$, which is equal to $M\backslash A^c$.
Define ${\cal B}(M)\backslash A:={\cal B}(M\backslash A)$, and so on.

The next lemma follows from the definition of contraction.

\begin{lem}\label{lem:contract}
For a matroid $M$ and $A\subseteq E$,
$${\cal B}(M)\cap (A,r(A))_= = {\cal B}(M|A)\oplus {\cal B}(M/A).$$
\end{lem}

\begin{proof}
Note that ${\cal B}(M|A)=({\cal B}(M)\cap (A,r(A))_=)|A$ and ${\cal B}(M/A)=({\cal B}(M)\cap (A,r(A))_=)|A^c.$
${\cal B}(M)\cap (A,r(A))_=$ is a matroid base system since every face of a matroid base system is a matroid base system.
Let $M'$ be the matroid ${\cal B}(M)\cap (A,r(A))_=$ with the rank function $r'$.
Since ${\cal B}(M')\subseteq (E, r(E))\cap (A,r(A))_=$, we have $r'(A)+r'(A^c)=r'(E)$ by Lemma \ref{lem:note}.
We have ${\cal B}(M)\cap (A,r(A))_= = {\cal B}(M|A)\oplus {\cal B}(M/A)$ since $r'(A)+r'(A^c)=r'(E)$.
\end{proof}

The base system ${\cal B}(M^*)$ of the dual matroid $M^*$ is given by $\{B^c|B \in {\cal B}(M)\}$. It is known that the connectivity of ${\cal B}(M^*)$ is equivalent to that of ${\cal B}(M)$.


\begin{lem}\label{lem:dualcont}
For a matroid $M$ and $A\subseteq E$,
$M^*|A^c$ is connected if and only if $M/A$ is connected.
\end{lem}

\begin{proof}
$M^*|A^c=M^*\backslash A=(M/A)^*$. 
Note that the connectivity of $(M/A)^*$ is equal to that of $M/A$.
\end{proof}


\begin{thm}\label{thm:facetbase}
For a connected matroid $M$ on $E$, the following are equivalent.

(a) $(A,r(A))_\leq$ is a facet-defining inequality of the base system ${\cal B}(M)$.

(b) ${\cal B}(M) \cap (A,r(A))_=$ has the two connected components $A$ and $A^c$.

(c) $M|A$ and $M/A$ are both connected.

(d) $(A,r(A))_\leq$ is a facet-defining inequality of the independence system ${\cal I}(M)$
 and $(A^c,r^*(A^c))_\leq$ is a facet-defining inequality of the independence system ${\cal I}(M^*)$ of the dual matroid $M^*$.

\end{thm}

\begin{proof}

[(a) $\leftrightarrow$ (b)] Since $(A,r(A))_\leq$ is a valid inequality to ${\cal B}(M)$, ${\cal B}(M) \cap (A,r(A))_=$ is a proper face of ${\cal B}(M)$. 
Any facet of the base system of the connected matroid is a matroid base system of dimension $|E|-2$. By the relation between the dimension of a matroid base system and the number of its connected components,
 $(A,r(A))_\leq$ is a facet-defining inequality of the base system if and only if ${\cal B}(M) \cap (A,r(A))_=$ has two connected components.

[(b) $\leftrightarrow$ (c)] By Lemma \ref{lem:contract}, (b) is equivalent to (c).

[(c) $\leftrightarrow$ (d)] Note that $A$ is a flat of $M$ when $M/A$ is connected.
Moreover $A^c$ is a flat of $M$ when $M|A$ is connected.
$(A,r(A))_\leq$ is a facet-defining inequality of the independence system of the matroid $M$ if and only if $A$ is a flat and $M|A$ is connected by Theorem \ref{thm:bifacet}.
$(A^c,r^*(A^c))_\leq$ is a facet-defining inequality of the independence system of the dual matroid $M^*$ if and only if $A^c$ is a flat and $M^*|A^c$ is connected by Theorem \ref{thm:bifacet}.  By Lemma \ref{lem:dualcont}, $M^*|A^c$ is connected if and only if $M/A$ is connected.
\end{proof}

Consequently, when $(A,r(A))_\leq$ is a facet-defining inequality of ${\cal B}(M)$, $A$ is a flat. We call such a flat a {\it facet-defining flat} of ${\cal B}(M)$. Note that $|A|>r(A)$ for any facet-defining inequality $(A,r(A))_\leq$.

Note that the decomposition in terms of matroid base polytopes is totally different from the decomposition by connected components.


We give an example of a facet of a matroid independence system that is not a facet of the matroid base system.

\begin{exa}\label{exa:csmis}
  Let $E=\{a,b,c,d,e,f\}$.
 Consider the matroid base system ${\cal
    B}(M)=(E,3)_= \cap (\{a,b,c,d\},2)_\leq\cap (\{a,b,e,f\},2)_\leq$ (the left of Figure \ref{fig:nonbase}). 
  However, the independence system $(E,3)_\leq\cap (\{a,b,c,d\},2)_\leq\cap (\{a,b,e,f\},2)_\leq$ does not satisfy submodularity 
$$r'(\{a,b,c,d\}) + r'(\{a,b,e,f\}) \geq r'(\{a,b,c,d,e,f\})+r'(\{a,b\})$$
as in Theorem \ref{thm:facetsub} since $r'(\{a,b\})=2$.
Therefore it is not a matroid independence system. 
The matroid independence system corresponding to ${\cal B}(M)$ has the rank function
  $r(A)=\max\{|B\cap A|\ |B\in {\cal B}(M)\}$.
 Therefore $r(\{a,b\})=1$.
 We have $${\cal I}(M)=(\{a,b,c,d\},2)_\leq \cap
  (\{a,b,e,f\},2)_\leq \cap (\{a,b\},1)_\leq, \mbox{and }$$
  $${\cal B}(M)={\cal I}(M)\cap (E,3)_= =(E,3)_=\cap (\{a,b,c,d\},2)_\leq \cap
  (\{a,b,e,f\},2)_\leq \cap (\{a,b\},1)_\leq.$$
  ${\cal I}(M)\cap (\{a,b\},1)_=$ is a facet of the independence system ${\cal I}(M)$ of matroid $M$.
Note that $\{a,b\}$ is a flat and $M|{\{a,b\}}$ is connected but not a facet-defining inequality of ${\cal B}(M)$ 
since $M/{\{a,b\}}$ is not connected. 
In fact, ${\cal B}(M)\cap (\{a,b\},1)_=$ has three connected components. Therefore ${\cal B}(M)\cap (\{a,b\},1)_=$ is not a facet of ${\cal B}(M)$.
The right of Figure \ref{fig:nonbase} depicts the dual matroid $M^*$, satisfying 
${\cal B}(M^*)=(E,3)_=\cap (\{c,d\},1)_\leq \cap (\{e,f\},1)_\leq$. Note that
${\cal B}(M^*)|\{c,d,e,f\}=(\{c,d\},1)_=\cap (\{e,f\},1)_=$ is not connected.

  \begin{figure}[ht]
    \begin{center}
      \leavevmode
\includegraphics[width=3.5cm]{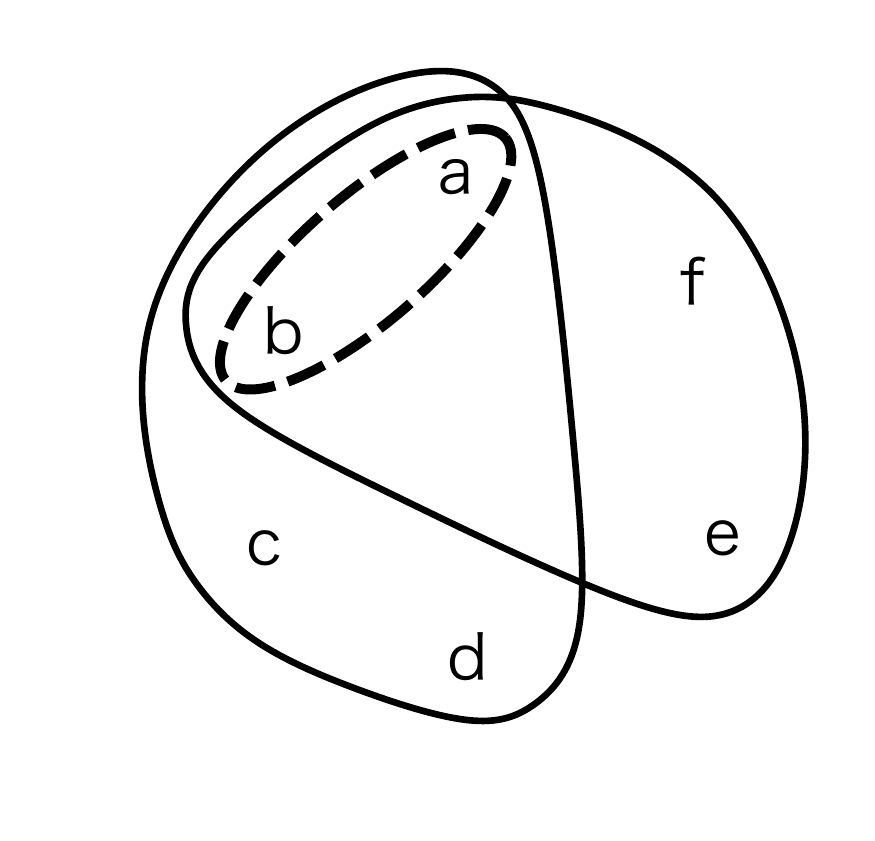}
\includegraphics[width=3.5cm]{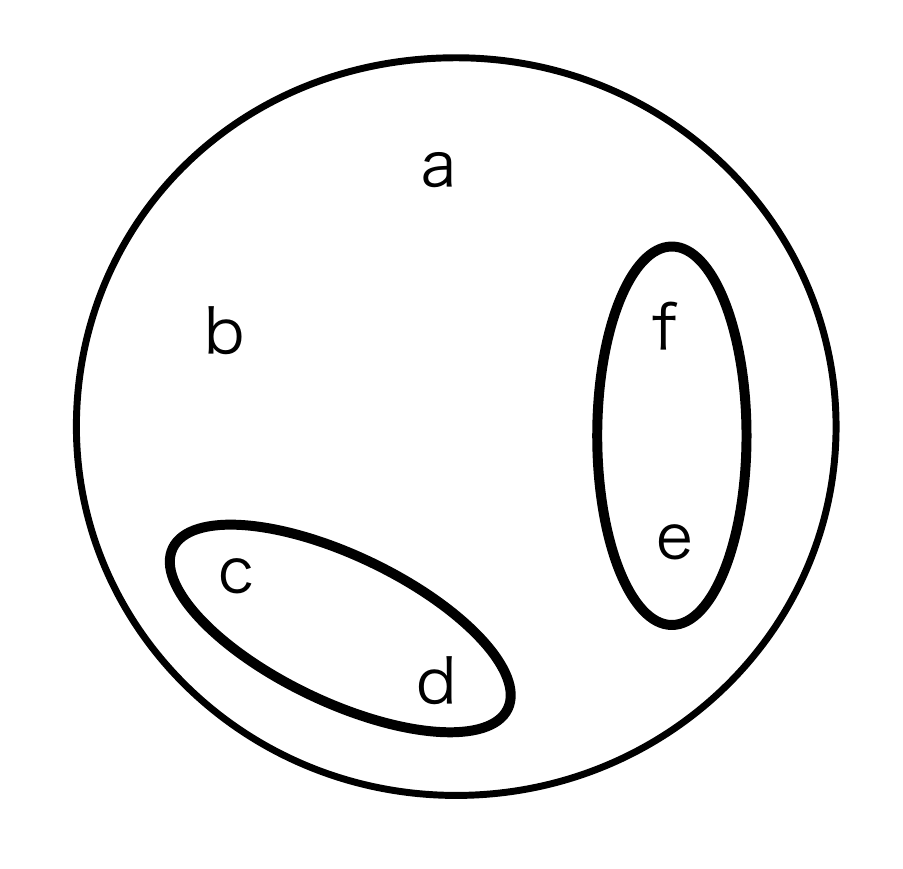}
    \end{center}
    \caption{${\cal B}(M)$ and ${\cal B}(M^*)$. ${\cal B}(M)\cap (\{a,b\},1)_=$ is a facet of the independence system ${\cal I}(M)$ but not a facet of the base system ${\cal B}(M)$.}\label{fig:nonbase}
  \end{figure}
\end{exa}

\section{The weak-map order and polytopal decompositions of matroid base systems}\label{sec:weak}

\subsection{The weak-map order and polytopal decompositions of matroid base systems of general rank}\label{subsec:weak}

Consider a partial ordering, called the {\it weak-map order},
 on the matroid base systems of rank $r$ on $E$ (See \cite{Crapo70}).
 For matroids $M_1$ and $M_2$ of the same rank on $E$, $M_1 \succeq M_2$ in the weak-map order  is defined by ${\cal B}(M_2)\subseteq {\cal B}(M_1)$.
 The maximum element of the weak-map order on all the matroids of rank $r$ on $E$ is the uniform matroid of rank $r$, denoted by $U_{r,|E|}$.  




The proof of the next lemma is straightforward.

\begin{lem}\label{lem:r1r2}
Assume that a matroid $M_1$ on $E$ has a rank function $r_1$ and a matroid $M_2$ on $E$ has a rank function $r_2$ with $r_1(E)=r_2(E)$.
${\cal B}(M_2)\subseteq {\cal B}(M_1)$ if and only if $r_2(A) \leq r_1(A)$ for all $A\subseteq E$.
\end{lem}



Next, we consider a polytopal decomposition of a matroid base system.

As already defined, a matroid base polytope is said to be decomposable if
it has a polytopal decomposition which consists of matroid base polytopes.
Next we rewrite a decomposition in terms of matroid base systems.

\begin{df}\label{df:decomp}
A matroid base system ${\cal B}(M)$ is {\it decomposable} if
there exist matroid base systems $\{{\cal B}_i\}_i$ satisfying the following conditions.

\begin{enumerate}
\item ${\cal B}(M)=\bigcup_i {\cal B}_i$,
\item For any $i$, the connected components of matroid $M_i$ corresponding to ${\cal B}_i$ are the same as those of $M$.
\item For any distinct $i$ and $j$, ${\cal B}_i\cap {\cal B}_j$ is a proper face of both ${\cal B}_i$ and ${\cal B}_j$, and there exists a hyperplane $(A,a)_=$ such that ${\cal B}_i\subseteq (A,a)_\leq$ and ${\cal B}_j\subseteq (A,a)_\geq$.
\item For any $i$ and each facet ${\cal B}$ of ${\cal B}_i$ with a facet-defining inequality $(A,a)_\leq$, ${\cal
    B}$ is a facet of ${\cal B}(M)$ or a facet of ${\cal B}_j$ for some unique $j$ with the facet-defining inequality $(A,a)_\geq$. 
\end{enumerate}
\end{df}

\begin{prop}
A matroid base system ${\cal B}(M)$ is decomposed into $\{{\cal B}(M_i)\}$ consisting of matroid base systems if and only if $B(M)$ has a polytopal decomposition $\{B(M_i)\}$ where $B(M)=\mbox{conv}({\cal B}(M))$ and $B(M_i)=\mbox{conv}({\cal B}(M_i))$ for any $i$.
\end{prop}

\begin{proof}
Assume that ${\cal B}(M)$ is decomposed into $\{{\cal B}(M_i)\}$ so that Conditions (a) to (d) in Definition \ref{df:decomp} are satisfied. Let $B(M)=\mbox{conv}({\cal B}(M))$ and $B(M_i)=\mbox{conv}({\cal B}(M_i))$. We first show $B(M)=\bigcup B(M_i)$. Every extreme point of $B(M)$ belongs to $\bigcup B(M_i)$ by Condition (a). Therefore, it suffices to show that $\bigcup B(M_i)$ is convex. By Condition (b), $B(M_i)$ and $B(M)$ have the same dimension for any $i$.
Even though $\bigcup\{B(M_i)\}$ may not be convex, we can consider the inequalities which define the boundary of $\bigcup\{B(M_i)\}$ on each connected component. Any inequality that defines the boundary of $\bigcup\{B(M_i)\}$ is a facet-defining inequality of $B(M)$ by Condition (d). Hence $\bigcup B(M_i)$ is convex. Therefore, we have $B(M)=\bigcup B(M_i)$. 

By Condition (d), $B(M_i)\cap B(M_j)=\mbox{conv}({\cal B}_i\cap {\cal B}_j)$ for distinct $i$ and $j$. By Condition (c), $B(M_i)\cap B(M_j)$ is a face of $B(M_i)$ and $B(M_j)$, and $B(M_i)\cap B(M_j)=\mbox{conv}({\cal B}_i\cap {\cal B}_j$. Therefore $B(M)$ has a polytopal decomposition $\{B(M_i)\}$.

Conversely, we assume that $B(M)$ has a polytopal decomposition $\{B(M_i)\}$. Let ${\cal B}(M)$ be the collection of (the supports of) the extreme points of $B(M)$, and ${\cal B}(M_i)$ be the collection of the extreme points of $B(M_i)$ for each $i$. Then, Conditions (a) to (d) in Definition \ref{df:decomp} are obviously satisfied.
\end{proof}

The next lemma is trivial but important when we consider the relation between the weak-map order and polytopal decompositions.

\begin{lem}\label{lem:divweak}
Every matroid base polytope belonging to a polytopal decomposition of a matroid base polytope $B(M)$ is smaller than the matroid base polytope $B(M)$ with respect to weak-map order.
\end{lem}

\begin{lem}\label{lem:simple}
A loopless matroid base system ${\cal B}(M)$ is decomposable if and only if the base system of the simplified matroid is decomposable.
\end{lem}

\begin{proof}
Assume that a matroid base system ${\cal B}(M)$ is decomposed into $\{{\cal B}(M_i)\}$. It suffices to show that $x, y\in E$ are parallel on $M_i$ when $x, y\in E$ are parallel on $M$ for $M_i$ in the decomposition with its rank function $r_i$. Since $M_i$ has no loop, $0< r_i(\{x,y\}) \leq r(\{x,y\})$. Therefore $r_i(\{x,y\}) =1$. 

The converse direction is obvious.
\end{proof}

Therefore we have only to consider simple matroids.

Next we consider the relation between decomposability and matroid connectivity. 
\begin{lem}
A matroid base system ${\cal B}(M)$ is not decomposable if and only if the matroid base system of no connected component of the matroid $M$ is decomposable.
\end{lem}

\begin{proof}
When the base system of some connected component of $M$ is decomposable, we can make a decomposition of ${\cal B}(M)$ in terms of the decomposition of such a connected component.

Conversely, assume that a matroid base system ${\cal B}(M)$ with its rank function $r$ is decomposed into ${\cal D}=\{{\cal B}(M_i)\}$. Let $\{E_k\}$ be the partition consisting of the separators induced by the connected components, which are the underlying sets of the connected components of $M$. We can take $k$ and $l$ so that ${\cal B}(M|E_k) \neq {\cal B}(M_l|E_k)$. $(A,r_l(A))_=$ is a facet-defining equality of $M_l$ with its rank function $r_l$ and ${\cal B}(M_j)\subseteq (A,r_l(A))_\geq$ for some $M_j\in {\cal D}$. Therefore $(A,r_l(A))_=$ is also a facet-defining equality of some $M_l|E_k$.
Therefore $\{{\cal B}(M_i|E_k) |\ {\cal B}(M_l|E_k^c) = {\cal B}(M_i|E_k^c), {\cal B}(M_i)\in {\cal D}\}$ is a decomposition of ${\cal B}(M|E_k)$.
\end{proof}

Therefore we have only to consider connected matroids when we consider decomposability.
Chatlain and Alfonsin\cite{Chatelain09} investigated the relation between decomposability of a matroid base polytope and its connected components.

\subsection{2-decomposability of matroid base systems}\label{sec:2decomp}

As already defined, when a matroid base polytope has a
polytopal decomposition which consists of exactly two matroid base polytopes, such a polytope is said to be {\it 2-decomposable}. 

We can rewrite this in terms of matroid base systems as follows.

\begin{df}\label{df:2decomposable}
For a matroid $M$, ${\cal B}(M)$ is 2-decomposable if there exists a
hyperplane $(A,a)_=$ such that 
\begin{enumerate}
\item ${\cal B}(M)\cap (A,a)_\geq $ is a matroid base system,
\item ${\cal B}(M)\cap (A,a)_\leq$ is a matroid base system,
\item ${\cal B}(M)\cap (A,a)_>$ is not empty, and
\item ${\cal B}(M)\cap (A,a)_<$ is not empty. 
\end{enumerate}
\end{df}

\begin{thm}\cite{Kashiwabara01}\label{thm:two}
  A matroid base system ${\cal B}(M)$ can be decomposed into two matroid base systems if and only if
 there exists a hyperplane $(A,r)_=$ such that 
\begin{itemize}
\item ${\cal B}(M)\cap (A,r)_=$ is a matroid base system,
\item ${\cal B}(M)\cap (A,r)_>$ and ${\cal B}(M)\cap (A,r)_<$ are non-empty.
\end{itemize}
\end{thm}

Kim\cite{Kim} also give a necessary and sufficient condition for 2-decompsability.
We give an example of a 2-decomposable matroid.

\begin{exa}\label{exa:2decomp}
  Let $E=\{a,b,c,d,e\}$.
 Consider the matroid independence system
  ${\cal I}(M)=(E,3)_\leq\cap (\{a,b,c\},2)_\leq$, illustrated in the
  first figure of Figure \ref{fig:m}. Then ${\cal B}(M)={\cal I}(M)\cap
  (E,3)_= = (E,3)_=\cap (\{a,b,c\},2)_\leq$.
 Even by adding $(\{c,d,e\},2)_=$ into the representation,
 the independence system ${\cal I}(M)\cap (\{c,d,e\},2)_\leq=(E,3)_\leq\cap
  (\{a,b,c\},2)_\leq\cap (\{c,d,e\},2)_\leq$,
 illustrated in the second figure, still satisfies submodularity as in Theorem  \ref{thm:facetsub}.
  By Theorem \ref{thm:two}, the base system ${\cal B}(M)$ of this matroid is divided into two base systems ${\cal B}(M_1)$
  and ${\cal B}(M_2)$ by hyperplane $(\{c,d,e\},2)_=$.
  The base system of one matroid $M_1$ is
$${\cal B}(M_1)={\cal B}(M)\cap (\{c,d,e\},2)_\leq=(E,3)_=\cap (\{a,b,c\},2)_\leq \cap (\{c,d,e\},2)_\leq,$$
  illustrated in the second figure of Figure \ref{fig:m}.
  The base system of the other matroid $M_2$ is 
  \begin{eqnarray*}
{\cal B}(M_2)&=&{\cal B}(M)\cap (\{c,d,e\},2)_\geq=(E,3)_=\cap (\{a,b,c\},2)_\leq\cap  (\{c,d,e\},2)_\geq\\
&=&(E,3)_=\cap (\{a,b,c\},2)_\leq \cap (\{a,b\},1)_\leq=(E,3)_=\cap (\{a,b\},1)_\leq,
  \end{eqnarray*}
 illustrated in the third figure of Figure \ref{fig:m}, since $(\{a,b\},1)_\leq\subseteq (\{a,b,c\},2)_\leq$
  by Lemma \ref{lem:facetinc}.
 The matroid base system corresponding to the cross-section by the hyperplane $(\{c,d,e\},2)_=$ is 
  \begin{eqnarray*}
{\cal B}(M_1\cap M_2)&=&{\cal B}(M)\cap (\{c,d,e\},2)_==(E,3)_=\cap (\{a,b,c\},2)_\leq\cap  (\{c,d,e\},2)_=\\
&=&(\{a,b\},1)_=\cap (\{a,b,c\},2)_\leq \cap (\{c,d,e\},2)_= = (\{a,b\},1)_=\cap (\{c,d,e\},2)_=
  \end{eqnarray*}
 since $(\{a,b\},1)_\leq \cap (\{c,d,e\},2)_\leq
  \subseteq (E,3)_\leq $. ${\cal B}(M_1\cap M_2)$ is illustrated in the last figure of Figure \ref{fig:m}.

  \begin{figure}[ht]
    \begin{center}
      \leavevmode
\includegraphics[width=33mm]{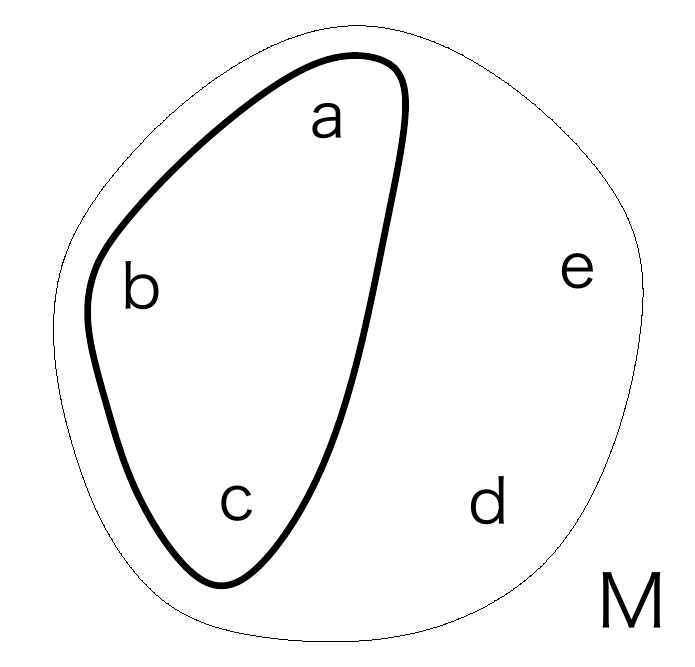}
\quad\includegraphics[width=33mm]{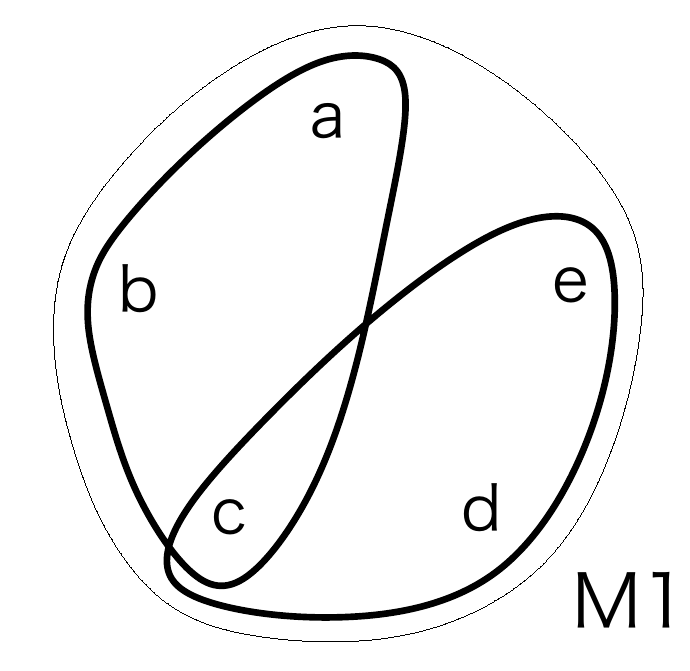}
\quad\includegraphics[width=33mm]{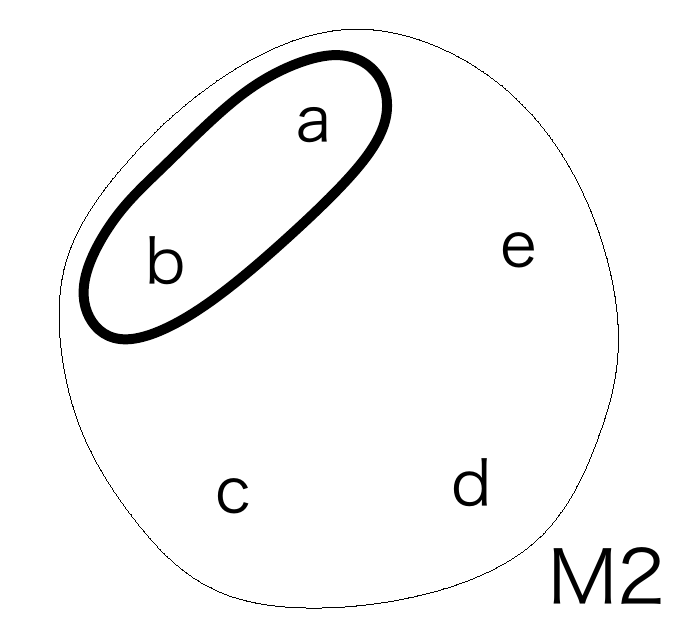}
\quad\includegraphics[width=33mm]{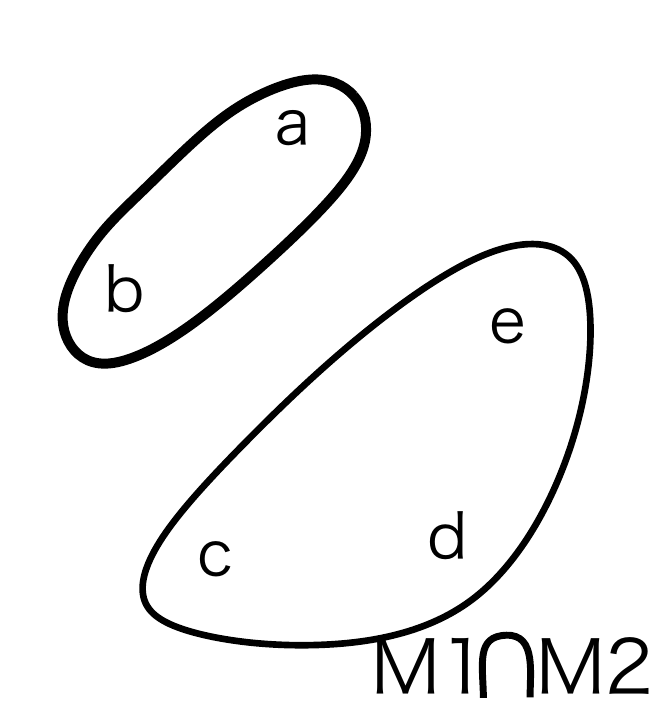} 
    \end{center}
    \caption{Dividing a matroid base system ${\cal B}(M)$ into two matroid base systems ${\cal B}(M_1)$ and ${\cal B}(M_2)$ with intersection ${\cal B}(M_1)\cap {\cal B}(M_2)$.}\label{fig:m}
  \end{figure}

\end{exa}

\section{Matroid base systems of rank 3}\label{sec:rank3}

In this section, we focus on matroid base systems of rank 3 only.
Matroid base systems of rank 3 will be used as important examples in Section \ref{sec:further}.


\subsection{Representation of matroid base systems of rank 3}


\begin{lem}\label{lem:ranktwoes}
For a matroid of rank 3, for any flat $F_1$ of rank 2 and any set $F_2$ of rank 2 with $F_2-F_1\neq \emptyset$, $F_1\cap F_2$ is of rank at most 1. Especially, 
for a matroid of rank 3, for any distinct flats $F_1, F_2$ of rank 2, $F_1\cap F_2$ is of rank at most 1. Moreover, $M$ is simple, $|F_1\cap F_2|\leq 1$.
\end{lem}

\begin{proof}
Since $F_1$ is a flat and $F_2-F_1\neq \emptyset$, $F_1\cup F_2$ has rank 3. Therefore, $F_1\cap F_2$ is of rank at most 1 by submodularity $r(F_1)+r(F_2)\geq r(F_1\cup F_2)+r(F_1\cap F_2)$.
\end{proof}




\begin{lem}\label{lem:facet2comp}
Consider a connected matroid $M$ of rank 3.
Let $A$ be a flat of rank 1.
Then $(A,1)_\leq$ is not a facet-defining inequality of ${\cal B}(M)$ if and only if ${\cal B}(M) \cap (A^c,2)_\leq$ has 
exactly two connected components.
\end{lem}

\begin{proof}
We have ${\cal B}(M) \cap (A^c,2)_\leq={\cal B}(M) \cap (A,1)_\geq={\cal B}(M) \cap (A,1)_=$ since the rank of $A$ is 1. Therefore the statement follows from Theorem \ref{thm:facetbase}.

\end{proof}

\begin{cor}\label{cor:facet2cond}
Consider a connected matroid $M$ of rank 3 on $E$.
Let $A\subseteq E$ be a flat of rank 1.
Then $(A,1)_\leq$ is a facet-defining inequality of ${\cal B}(M)$ if and only if 
there exists no pair of flats $F_1, F_2$ of rank 2 such that 
$A = F_1\cap F_2$ and $F_1\cup F_2 =E$.
\end{cor}

\begin{proof}
Assume that $A=F_1\cap F_2$, and $F_1\cup F_2=E$. Let $r$ be the rank function of $M$. 
It suffices to show that the matroid ${\cal B}(M)\cap (A,1)_=$ has three connected components by Lemma \ref{lem:facet2comp}. Let $r'$ be the rank function of ${\cal B}(M)\cap (A,1)_=$. Note that ${\cal B}(M)\cap (A,1)_=={\cal B}(M)\cap (A^c,2)_=$ by Lemma \ref{lem:note}.
Therefore, on ${\cal B}(M)\cap (A^c,2)_=$, $A^c \cap F_1$ has rank 1 because of the submodularity $r'(F_1)+r'(A^c)\geq r'(F_1 \cap A^c)+r'(E)$. Similarly, $A^c\cap F_2$ has rank 1 on ${\cal B}(M)\cap (A^c,2)_=$. Since $r'(A^c \cap F_1)+r'(A^c \cap F_2)=r'(A^c)$,  ${\cal B}(M)\cap (A^c,2)_\leq$ has three connected components $A, A^c\cap F_1$ and $A^c\cap F_2$.

Conversely, assume that $(A,1)_\leq$ is not a facet-defining inequality.
Then ${\cal B}(M)\cap (A^c,2)_\leq$ has exactly two connected components of rank 1 on $A^c$ by Lemma \ref{lem:facet2comp}. Let $A_1$ and $A_2$ be such connected components.
Note that $A$ is also a flat of rank 1 on ${\cal B}(M)\cap (A^c,2)_\leq$.

We first show that $F_1:=A\cup A_1$ and $F_2:=A\cup A_2$ have both rank 2 on $M$.
Note that $r(A_1)=r'(A_1)=1$ and $r(A_2)=r'(A_2)=1$ since $A_1, A_2 \subseteq A^c$.
$F_2=A\cup A_2$ has rank 2 on $M$ because $A$ is a flat of rank 1 and submodularity  $r(F_2)\leq r(A)+r(A_2)$. $F_1=A_2^c$ has rank 2 on $M$ becuase of submodularity $r(A)+r(A_1)\geq r(F_1)$. Note that $r(F_1)>r(A)=1$ since $A$ is a flat on $M$.
The only remaining part to show is that $F_1$ and $F_2$ are flats on $M$.
On the contrary, suppose that there exists $F\subseteq E$ such that $A\cup A_1\subsetneq F$ and $r(F)=2$. Then $F\cap (A_1\cup A_2)$ has rank 1 by submodularity, which contradicts that $A_1$ is a flat of rank 1 on $M$.
\end{proof}

\begin{prop}
For a connected simple matroid, every flat $F$ of rank 2 is facet-defining when $(F,2)_\leq$ is non-trival, i.e. $|F|\geq 3$.
\end{prop}

\begin{proof}
We first show the restriction $M|F$ of $M$ to $F$ is connected. 
On the contrary, suppose that $M|F$ has two connected components $A_1$ and $A_2$ of rank 1 on $M$. Since $|F|\geq 3$, $|A_1|\geq 2$ or $|A_2|\geq 2$, which contradicts that $M$ is simple.


Therefore, ${\cal B}(M)\cap (F,2)_=$ has exactly two connected components.
Hence  $(F,2)_\leq$ is a facet-defining inequality by Theorem \ref{thm:facetbase}.
\end{proof}


\begin{lem}\label{lem:ff}
Consider a connected matroid $M$ of rank 3 on $E$ with a rank function $r$.
For a facet-defining inequality $(F_2,2)_\leq$ and $F_1\subseteq E$ such that $r(F_1)=1$ and $F_1\cap F_2\neq \emptyset$, we have
$F_1 \subseteq F_2$ and $|F_2-F_1|\geq 2$.
\end{lem}

\begin{proof}
When $F_2$ and $F_1$ do not satisfy $F_1 \subseteq F_2$, we have $r(F_1\cup F_2)=3$ since $F_2$ is a flat and the rank of $M$ is 3.  Since $F_1\cap F_2\neq \emptyset$ and $M$ is connected, we have $r(F_1\cap F_2)=1$.
Therefore it contradicts submodularity $r(F_1)+r(F_2)\geq r(F_1\cap F_2)+r(F_1\cup F_2)$.
Hence $F_1\subseteq F_2$. 

When $|F_2-F_1|=1$ and $F_1\subseteq F_2$, matroid base system ${\cal B}(M)\cap (F_2,2)_=$ has a connected component $F_2-F_1$. Since $(F_1,1)_\leq \subseteq (F_2,2)_\leq$ by Lemma \ref{lem:facetinc}, $(F_2,2)_\leq$ is not a facet-defining inequality.
\end{proof}

\subsection{2-decomposability of matroids of rank 3}


In this subsection, we consider the 2-decomposability of a matroid base system of rank 3. 

Note that a connected matroid base system of rank 3 is 2-decomposable by $(A,1)_=$ if and only if it is 2-decomposable by $(A^c,2)_=$ since $(E,3)_=\cap (A,1)_= =(E,3)_=\cap(A^c,2)_=$ by Lemma \ref{lem:note}.

\begin{thm}\label{thm:2div}
A connected simple matroid base system ${\cal B}(M)$ of rank 3 is 2-decomposable by $(A,2)_=$
 if and only if it satisfies the following conditions.

(1) $|A^c|\geq 2$,

(2) the rank of $A$ is 3,

(3) for any facet-defining inequality $(F,2)_\leq$ of ${\cal B}(M)$, one of the following is satisfied.
(3-1) $|A\cap F|\leq 1$,
(3-2) $F\subseteq A$,
(3-3) $A\cup F=E$.
\end{thm}

\begin{proof}
Assume that (1),(2), and (3) are satisfied.
We show that ${\cal B}(M)\cap (A,2)_\leq$ is a matroid base system.
Let $r'$ be the rank function of ${\cal B}(M)\cap (A,2)_\leq$. Note that $r'(X) \leq r(X)$ always holds.
We separate the cases according to (3-1), (3-2) and (3-3) to show submodularity $r'(A)+r'(F)\geq r'(A\cup F)+r'(A\cap F)$ for any facet-defining flat $F$ of rank 2. 
In the case (3-1), $r'(F)=r(F)=2$ since $r'(F-A)=2$. So the submodularity holds in this case. 
In the case (3-2), the submodularity holds trivially. 
To consider the case (3-3), we assume $F\cup A=E$. Since $r'(A)=2$ and $r'(F\cap A)\leq r(F\cap A)$, it suffices to show $r(F\cap A)\leq 1$. When $r(F\cap A)=2$, there exists $B\in {\cal B}(M)$ such that $B\cap (F\cap A)|=2$. Since $|B|=3$ and $F\cup A=E$, $|F\cap B|=3$ or $|A\cap B|=3$, which contradicts $r(F)=2$ and $r(A)=2$. Therefore $r(F\cap A)=1$. We have completed the case (3-3).

By Theorem \ref{thm:facetsub}, ${\cal B}(M)\cap (A,2)_\leq$ is a matroid base system. So ${\cal B}(M)\cap (A,2)_=$ is a matroid base system.

By Condition (1) and the simpleness of the matroid, there exists a base $B\in {\cal B}(M)$ such that $|B\cap A|\leq 1$. Therefore ${\cal B}(M)\cap (A,2)_<$ is not empty. 
By Condition (2), ${\cal B}(M)\cap (A,2)_>$ is not empty.
Therefore, ${\cal B}(M)$ is 2-decomposable by Theorem \ref{thm:two}.

We consider the converse direction.
Assume that there exists a flat $A$ of rank 2 and which satisfies neither of (3-1), (3-2), nor (3-3). Then there exists a facet-defining inequality $(F,2)_\leq$ such that
$|A\cap F|> 1$, $F-A\neq \emptyset$, $A\cup F\neq E$. Since $F-A\neq\emptyset$ and $A$ is a flat of rank 2, we have $r'(A\cup F)=3$. Since $|A\cap F|> 1$ and $A\cup F\neq E$, we have $r'(A\cap F)\geq 1$.
Therefore ${\cal B}(M)\cap (A,2)_\leq$ does not satisfy submodularity $r'(A)+r'(F)\geq r'(A\cup F)+r'(A\cap F)$. 
When the rank of $A$ is 2, ${\cal B}(M)\cap (A,2)_>$ is empty. When $|A^c|\leq 1$, ${\cal B}(M)\cap(A,2)_\leq$ is not connected. Therefore when one of Conditions (1),(2) and (3) is not satisfied, ${\cal B}(M)$ is not 2-decomposable.
\end{proof}

\begin{cor}\label{cor:twopoints}
Consider a connected simple matroid base system ${\cal B}(M)$ of rank 3. If there exist $x,y\in E$ with $x\neq y$ such that $\{F\in {\cal F}^2(M)|x\in F\}=\{F\in {\cal F}^2(M)|y\in F\}$ and the rank of $\{x,y\}^c$ is 3, then it is 2-decomposable by $(\{x,y\},1)_=$ where ${\cal F}^2(M)$ denotes the set of facet-defining flats of rank 2.
\end{cor}

\begin{proof}
It suffices to show that $\{x,y\}^c$ satisfies Conditions (1), (2), and (3) of Theorem \ref{thm:2div}.
Condition (1) follows from $|\{x,y\}|=2$. Condition (2) follows from the assumption $r(\{x,y\}^c)=3$. We check the Condition (3). Note that $|F\cap\{x,y\}|\neq 1$ holds for a facet-defining flat $F$ of rank 2 by the assumption. Condition (3-3) $F\cup \{x,y\}^c=E$ holds when a facet-defining flat $F$ of rank 2 satisfies $F\supseteq \{x,y\}$. Condition (3-2) holds when a facet-defining flat $F$ of rank 2 satisfies $F\cap \{x,y\}=\emptyset$.
\end{proof}

\begin{exa}
Let $E=\{a,b,c,d,e,f\}$. Consider the following matroid base system.
$${\cal B}(M)=(E,3)_=\cap (\{a,b,c\},2)_\leq\cap (\{c,d,e,f\},2)_\leq.$$
Since $\{e,f\}$ satisfies the conditions in Corollary \ref{cor:twopoints}, ${\cal B}(M)$ is 2-decomposable by $(\{e,f\},1)_=$ into ${\cal B}(M_1)$ and ${\cal B}(M_2)$.
$${\cal B}(M_1)=(E,3)_=\cap (\{a,b,c\},2)_\leq\cap (\{c,d,e,f\},2)_\leq \cap (\{e,f\},1)_\leq,$$
$${\cal B}(M_2)=(E,3)_=\cap (\{c,d,e,f\},2)_\leq \cap (\{a,b,c,d\},2)_\leq.$$
\end{exa}

\begin{cor}\label{cor:co}
Consider a connected simple matroid base system of rank 3 with $|E|\geq 5$. Define the graph on $E$ so that $\{x,y\}$ is an edge of the graph if and only if there exists no facet-defining flat of rank 2 including $\{x,y\}$. If the graph has a 3-cycle on $\{x,y,z\}$, the matroid base system is 2-decomposable by $(\{x,y,z\},2)_=$.
\end{cor}

\begin{proof}
Let $\{x,y,z\}$ have such a 3-cycle. We check that $A=\{x,y,z\}$ satisfies the conditions in Theorem \ref{thm:2div}. Condition (1) is satisfied because of $|E| \geq 5$. Conditions (2) and (3-1) follow from that $\{x,y,z\}$ intersects any facet-defining flat of rank 2 in at most 1 element.
\end{proof}

\begin{exa}
Let $E=\{a,b,c,d,e,f\}$. Consider the following matroid base system.
$${\cal B}(M)=(E,3)_=\cap (\{a,b,c\},2)_\leq\cap (\{c,d,e\},2)_\leq\cap(\{e,f,a\},2)_\leq.$$
Since each of $\{b,d\},\{d,f\},\{f,b\}$ is contained in no facet-defining flat of rank 2, $\{b,d,f\}$ contains a 3-cycle satisfying the conditions in Corollary \ref{cor:co}.
Therefore the matroid base system is 2-decomposable by $(\{b,d,f\},2)_=$.
$${\cal B}(M_1)=(E,3)_=\cap (\{a,b,c\},2)_\leq\cap (\{c,d,e\},2)_\leq\cap(\{e,f,a\},2)_\leq\cap (\{b,d,f\},2)_\leq.$$
$${\cal B}(M_2)=(E,3)_=\cap (\{a,c\},1)_\leq.$$
\end{exa}

\begin{cor}\label{cor:notall}
For a connected simple matroid base system of rank 3, if there exist a flat $A$ of rank 2 and $x\in E-\bigcup\{F\in {\cal F}^2|A\cap F\neq \emptyset\}-A$ and $|E-A|\geq 3$, the matroid base system is 2-decomposable by $(A\cup x,2)_=$, where ${\cal F}^2$ is the set of the facet-defining flats of rank 2.
\end{cor}

\begin{proof}
We have only to show Conditions (1), (2), and (3) in Theorem \ref{thm:2div}.
Condition (1) follows from $|E-A|\geq 3$.
Condition (2) $r(A\cup x)=3$ follows from that $A$ is a flat of rank 2 and $x\notin A$.
For every facet-defining flat $F$ of rank 2, $|F\cap (A\cup x)|= |F\cap A|$ when $A\cap F\neq \emptyset$ since $x\notin F$. $r(F\cap A)\leq 1$ by the submodularity. $|F\cap A|\leq 1$ since a matroid is simple. When $A\cap F=\emptyset$, $|F\cap (A\cup x)|\leq 1$ holds.  Therefore we have (3-1) $|F\cap (A\cup x)|\leq 1$.
\end{proof}

\begin{exa}\label{exa:notall}
Let $E=\{a,b,c,d,e,f,g,h\}$.
The following example is a matroid base system which is 2-decomposable by Corollary \ref{cor:notall}.
$${\cal B}(M)=(E,3)_=\cap (abc,2)_\leq\cap (ade,2)_\leq\cap(afg,2)_\leq\cap(bdf,2)_\leq\cap (ceh,2)_\leq\cap (bgh,2)_\leq\cap (cfh,2)_\leq.$$
This matroid is depicted in Figure \ref{fig:notall}.
Corollary \ref{cor:co} cannot apply to this matroid base system. However, since $\{a,d,e\}$ and $h$ satisfy the conditions in Corollary \ref{cor:notall}, it is 2-decomposable by $(\{a,d,e,h\},2)_=$.
\begin{figure}
\begin{center}
\includegraphics{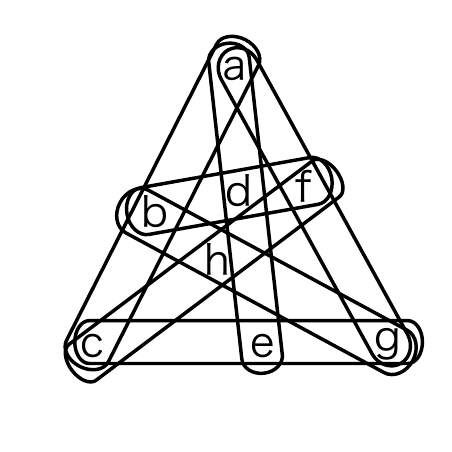}
\caption{The flats in the matroid in Example \ref{exa:notall}}\label{fig:notall}
\end{center}
\end{figure}
\end{exa}

The 2-decomposability of a connected simple matroid base system of small size is always decided by Corollaries \ref{cor:twopoints}, \ref{cor:co} and \ref{cor:notall}.
However, that of some connected simple matroid base system cannot be decided by the above corollaries as in the next example.

\begin{exa}
Let $E=\{a,b,c,d,e,f,g,h,i,j,k,l\}$.
Consider the following matroid base system on $E$ whose facet-defining flats of rank 2 are
$$adgi,bcei,abhj,cdfj,acl,bdl,egl,fhl,ijl,aek,bfk,cgk,dhk.$$

Since the matroid has too many facet-defining flats, we use two figures to represent them. 
The first figure in Figure \ref{fig:2divisible} indicates the facet-defining flats of rank 2 containing $i$ or $j$. The second figure indicates the facet-defining flats of rank 2 containing $k$ or $l$. 
The matroid has facet-defining flat $\{i,j,l\}$ other than those depicted in the figures.

There exists no pair of elements as in Corollary \ref{cor:twopoints}. There exists no 3-cycle as in Corollary \ref{cor:co}. There exists no facet-defining flat as in Corollary \ref{cor:notall}. However, by Theorem \ref{thm:2div}, it is 2-decomposable by $(\{e,f,g,h,l\},2)_=$.
\begin{figure}
\begin{center}
\includegraphics{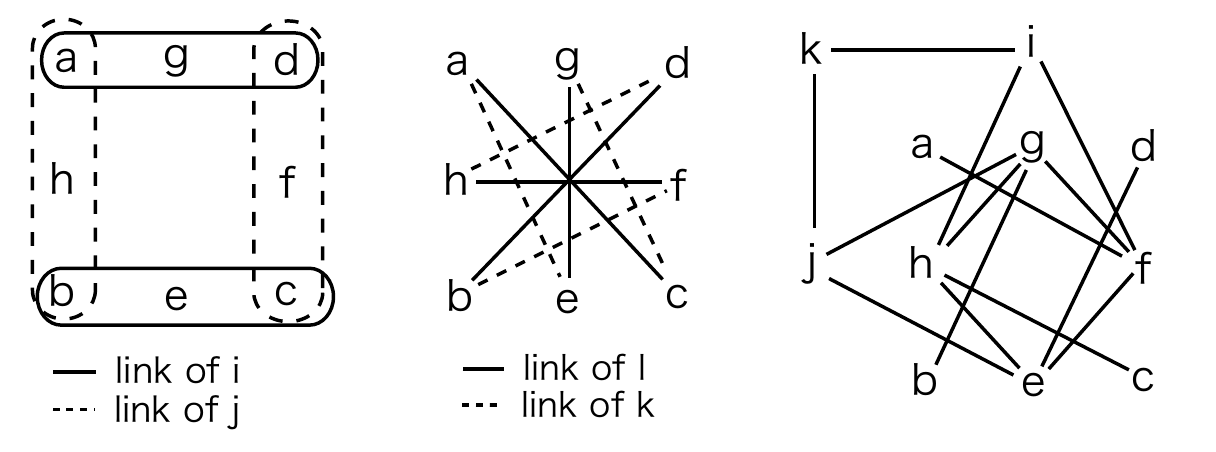}
\caption{The first two figures show the facet-defining flats. The last figure shows the graph defined in Corollary \ref{cor:co}}\label{fig:2divisible}
\end{center}
\end{figure}

\end{exa}



\subsection{Included matroids and 3-partitions}

Since we handle two or more matroids simultaneously, we name them to distinguish.
We call a connected simple matroid base system that we consider as a whole matroid an  {\it original matroid base system}.
In this section, any original matroid base system is assumed to be of rank 3.
A facet-defining flat of the original matroid base system is called an {\it original facet-defining flat}.

We fix an original matroid base system of rank 3. 
A connected matroid base system that is properly included in the original matroid base system is called an {\it included matroid base system}. Note that an included matroid base system may not be simple. A facet-defining flat of an included matroid base system that is not an original facet is called a {\it non-original facet-defining flat}.
For an original matroid base system ${\cal B}(M)$ with its rank function $r$, and its included matroid base system ${\cal B}(M')$ with its rank function $r'$, $r'(A)\leq r(A)$ holds for any $A\subseteq E$ by Lemma \ref{lem:r1r2}.

It is known that any binary matroid base system cannot include any included matroid base system as shown in Theorem \ref{thm:lucas}.

\begin{exa}\label{exa:seven}
The next example is a matroid base system on $E$ of size 7 that is neither binary nor 2-decomposable but not minimal with respect to the weak-map order.
Let $E=\{a,b,c,d,e,f,g\}$.
$${\cal B}(M)=(E,3)_=\cap (\{a,b,c\},2)_\leq\cap (\{a,d,e\},2)_\leq\cap(\{a,f,g\},2)_\leq\cap(\{b,d,f\},2)_\leq\cap (\{c,e,g\},2)_\leq.$$
Since any two facet-defining flats of rank 2 intersect in at most one element.
It is not binary since $M/\{c\}\backslash \{b,g\}=U_{2,4}$.
However, it is not 2-decomposable by Theorem \ref{thm:2div}.
$${\cal B}(M_1)=(E,3)_=\cap (\{a,b,c,d,e\},2)_\leq\cap (\{a,f,g\},2)_\leq\cap(\{b,d\},1)_\leq\cap(\{c,e\},1)_\leq$$
is an included matroid base system of $M$, shown in Figure \ref{fig:seven}. 
For example, $(\{a,b,c,d,e\},2)_\leq$ is a non-original facet-defining flat, and $(\{a,f,g\},2)_\leq$ is an original facet-defining flat.
\begin{figure}
\begin{center}
\includegraphics{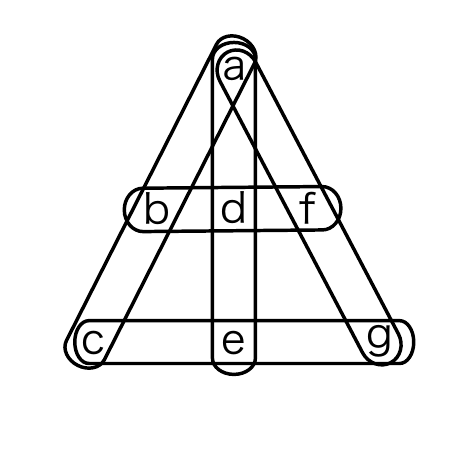}
\includegraphics{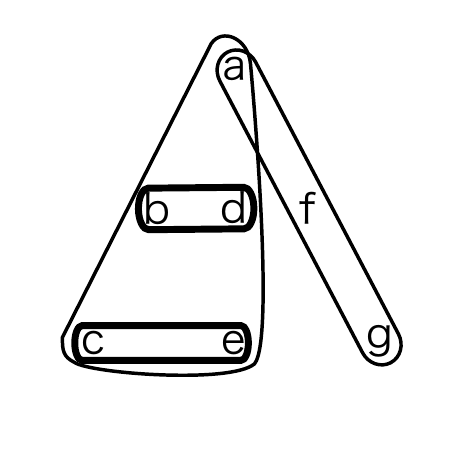}
\caption{The original matroid $M$ and an included matroid $M_1$ in Example \ref{exa:seven}}\label{fig:seven}
\end{center}
\end{figure}
\end{exa}

\begin{df}
For a connected simple matroid base system ${\cal B}(M)$ of rank 3 with its rank function $r$,
a partition $\{A_1, A_2, A_3\}$ of $E$ is called a {\it 3-partition} in $M$ if the following conditions are satisfied.
\begin{itemize}
\item $|A_1|, |A_2|, |A_3|\geq 2$,
\item there exists no facet-defining flat of rank 2 that intersects all of $A_1,A_2,A_3$,
\item $r(A_1\cup A_2)=r(A_2\cup A_3)=r(A_3\cup A_1)=3$.
\end{itemize}
\end{df}

A 3-partition $\{A_1,A_2,A_3\}$ corresponds to a matroid $(A_1,1)_=\cap (A_2,1)_=\cap (A_3,1)_=$.
When the matroid corresponding to a 3-partition in $M$ is a ridge of its included matroid $M'$, we say that ${\cal B}(M')$ has a 3-partition in $M$.

\begin{lem}\label{lem:2facets}
Consider a connected simple matroid base system ${\cal B}(M)$ of rank 3 which is not 2-decomposable and has an included matroid base system ${\cal B}(M')$.
Then ${\cal B}(M')$ has a 3-partition $\{A_1, A_2, A_3\}$ in $M$ such that $(A_1, 1)_\leq$ and $(A_1\cup A_2, 2)_\leq$ are non-original facet-defining flats.
\end{lem}

\begin{proof}
Let $r'$ be the rank function of $M'$.
Since the matroid base system ${\cal B}(M)$ is not 2-decomposable,
${\cal B}(M')$ has at least two non-original facet-defining flats.

Consider the case where every non-original facet-defining flat of ${\cal B}(M')$ has rank 1. Let $(A_1,1)_\leq$ be a non-original facet-defining flat of $M'$.
Then any other non-original facet-defining inequality $(A,1)_\leq$ of ${\cal B}(M')$ is disjoint from $A_1$ on $E$ since $M'$ is connected. Then ${\cal B}(M)\cap (A_1,1)_\leq$ is a matroid base system since ${\cal I}(M)\cap (A_1,1)_\leq$ is a matroid independence system by Theorem \ref{thm:facetsub}.
By Theorem \ref{thm:two}, ${\cal B}(M)$ is 2-decomposable, a contradiction.

So we can assume that ${\cal B}(M')$ has a non-original facet-defining inequality $(F_1,2)_\leq$.

Since the matroid base system ${\cal B}(M)$ cannot be 2-decomposable by non-original facet-defining equality $(F_1,2)_=$, ${\cal B}(M)\cap (F_1,2)_\leq$ is not a matroid base system. Therefore ${\cal B}(M')$ has another non-original facet-defining inequality than $(F_1,2)_\leq$. 

We show that there exists a non-original facet-defining inequality $(A_1,1)_\leq$ such that $A_1\subseteq F$. On the contrary, suppose that there exists no such inequality.
If any other non-original facet-defining inequality than $(F_1,2)_\leq$ is either of the following three types (a), (b) and (c), it is easy to show that the submodularity $r''(F_1)+r''(F_2)\geq r''(F_1\cap F_2)+r''(F_1\cup F_2)$ of ${\cal B}(M)\cap (F_1,2)_\leq$ holds where $r''$ is the rank function of ${\cal B}(M)\cap (F_1,2)_\leq$.

(a) non-original facet-defining inequality $(A,1)_\leq$ with $A\cap F_1=\emptyset$.

(b) non-original facet-defining inequality $(F,2)_\leq$ with $|F\cap F_1|\leq 1$.

(c) non-original facet-defining inequality $(F,2)_\leq$ with $F\cup F_1 =E$.

Then ${\cal B}(M)$ is 2-decomposable by $(F_1,2)_=$ at that time by Theorem \ref{thm:two}. 
Therefore there exists a flat $F_1$ of rank 2 such that $|F\cap F_1|\geq 2$ and $F\cup F_1\neq E$. Then $F\cap F_1$ is a flat of rank 1 by Corollary \ref{cor:facet2cond}. By letting $A_1=F \cap F_1$, $(A_1,1)_\leq$ is a non-original facet-defining inequality on ${\cal B}(M')$ such that $A_1\subseteq F_1$ and $|A_1|\geq 2$, a contradiction. 

We show that $\{A_1, A_2, A_3\}$ is a 3-partition in $M$ where $A_2=F_1-A_1$ and $A_3=F_1^c$.

Since $(A_1 \cup A_2,2)_\leq$ is a non-original facet-defining inequality, ${\cal B}(M)\cap (A_1\cup A_2,2)_>$ is non-empty. Therefore $A_1\cup A_2$ has rank 3 on the matroid base system ${\cal B}(M)$.
If $A_1\cup A_3$ has rank 2 on $M$, it contradicts Corollary \ref{cor:facet2cond} since $(A_1,1)_\leq$ is a non-original facet-defining inequality and $(A_1\cup A_2)\cup (A_1\cup A_3)=E$.
So $r(A_1\cup A_3)=3$. Since the rank of $A_1$ is 1 and the included matroid $M'$ is connected,
 $A_2\cup A_3$ has rank 3 on the matroid $M$.

Since $(A_1,1)_\leq$ is a non-original facet-defining inequality, we have $|A_1|\geq 2$.
Since $(A_1\cup A_2,2)_\leq$ is a non-original facet-defining inequality, $|A_2|\geq 2$.
Since $|A_3|=1$ contradicts the connectivity of the matroid $M$, we have $|A_3|\geq 2$.

${\cal B}(M)\cap (A_1,1)_=\cap (A_1\cup A_2,2)_= = {\cal B}(M)\cap (A_1,1)_=\cap (A_2,1)_=\cap (A_3,1)_=$ includes a matroid base system of rank 3 as a ridge of ${\cal B}(M')$. There exists no loopless matroid included in ${\cal B}(M)\cap (A_1,1)_=\cap (A_2,1)_=\cap (A_3,1)_=$ for any facet-defining flat $F$ of rank 2 that intersects all of $A_1,A_2,A_3$. Note that the ridge of $M'$ defined by $(A_1,1)_=$ and $(A_1\cup A_2,2)_=$ cannot have any loop by considering its dimension. Therefore there exists no original facet-defining flat of rank 2 that intersects all of $A_1,A_2,A_3$.
Therefore $\{A_1,A_2,A_3\}$ is a 3-partition in $M$.
\end{proof}

By this lemma, a connected simple matroid of rank 3 whose base system has an included matroid base system has a 3-partition $\{A_1,A_2,A_3\}$ in $M$ and therefore the non-connected matroid base system $(A_1,1)_=\cap (A_2,1)_=\cap (A_3,1)_=$ is included in the original matroid base system, which is a ridge of the included matroid base system.


Let $e(k)$ be $k \choose 2$, which is the number of edges of the complete graph of size $k$.
We define function $f$ as follows. Let $f(2k)=k\times (k-1)$ for even number $2k$ and $f(2k+1)=k^2$ for odd number $2k+1$.
$f(n)$ is equal to the minimum number of $e(n_1)+e(n_2)$ so that $n=n_1+n_2$.
In fact, $f(2k)=2\times e(k)=2\times \frac{1}{2} k\times (k-1)$ and $f(2k+1)=e(k)+e(k+1)=\frac{1}{2} k\times (k-1)+\frac{1}{2} k\times (k+1)=k^2$.

\begin{lem}\label{lem:hatonosu}
Consider a connected simple matroid $M$ of rank 3.
When $\{A_1,A_2,A_3\}$ is a 3-partition in $M$,
$e(|A_1|)+e(|A_2|)+e(|A_3|)$ is equal to or more than the sum of $f(|F|)$ over all facet-defining flats $F$ of rank 2.
\end{lem}

\begin{proof}
Consider the three complete graphs such that the sets of vertices are $A_1,A_2,A_3$.
Then the sum of the numbers of their edges are $e(|A_1|)+e(|A_2|)+e(|A_3|)$.
Any original facet-defining flat $F$ cannot intersect all of $A_1,A_2,A_3$ by the definition of a 3-partition.
For each $A_i$, each pair $\{x,y\}$ of $x,y\in A_i$ is included in at most one original facet-defining flat of rank 2 by Lemma \ref{lem:ranktwoes} and the simpleness of $M$.
Each facet-defining flat $F$ is included in one of $A_1\cup A_2$, $A_2\cup A_3$, or $A_3\cup A_1$.
Therefore the number of edges included in each $F$ is at least $f(|F|)$.
\end{proof}

\subsection{The graph induced from original facets}

In this subsection, we introduce a graph to investigate the existence of an included matroid.

\begin{df}
Consider a connected simple matroid $M$ of rank 3 as an original matroid.
For any disjoint sets $A_1$ and $A_2$ on $E$, we define the graph $g(A_1,A_2)$ as follows.
The set of vertices of $g(A_1,A_2)$ is $A_2$.
For $x,y\in A_2$, $\{x,y\}$ is an edge if and only if
 ${\cal B}(M)$ has a facet-defining flat $F$ of rank 2 such that $\{x,y\}\subseteq F$ and $|F\cap A_1|\geq 1$.
\end{df}

The graph $g(A_1,A_2)$, which is defined on $A_2$, is said to be {\it connected} if $A_2$ is a connected component of the graph.

\begin{lem}\label{lem:graphdis}
For a connected simple matroid $M$ of rank 3 and a 3-partition $\{A_1,A_2,A_3\}$ in $M$,
$g(A_1,A_3)$ and $g(A_2,A_3)$ have no common edge.
\end{lem}

\begin{proof}
Suppose that $\{x,y\}$ is a common edge of $g(A_1,A_3)$ and $g(A_2,A_3)$. By the definition of the graph, there exists  a facet-defining flat $F_1$ of rank 2 including $\{x,y\}$ and intersecting $A_1$. Moreover there exists a facet-defining flat $F_2$ of rank 2 including $\{x,y\}$ and intersecting $A_2$. By the definition of a 3-partition, there exists no facet-defining flat intersecting all of $A_1, A_2,$ and $A_3$. Therefore $F_1$ and $F_2$ are distinct. The intersection of two distinct facet-defining flats $F_1$ and $F_2$ of rank 2 has at most rank 1 by Lemma \ref{lem:ranktwoes}. Since $M$ is simple, we have $|F_1\cap F_2|\leq 1$. It contradicts that both $F_1$ and $F_2$ include $\{x,y\}$.
\end{proof}





Given an original matroid ${\cal B}(M)$, how do we find an included matroid base system?
By using Lemma \ref{lem:2facets}, we can try to find a 3-partition in $M$.
Consider the original matroid and some additional inequalities as non-original facet-defining inequalities obtained from the 3-partition.
To make an included matroid base system, what inequalities we should add further to such a matroid base system as restrictions?

When it satisfies $(F_1,1)_\leq$ and $(F_2,2)_\leq$ with neither 
$F_1\subseteq F_2$ nor $F_1\cap F_2=\emptyset$,
assume that it satisfies $(F_1\cup F_2,2)_\leq$.

For $(F_1,2)_\leq$ and $(F_2,2)_\leq$, when $|F_1\cap F_2|\geq 2$ and $r(F_1\cap F_2)=2$,
we have two ways to hold submodularity $r(F_1)+r(F_2) \geq r(F_1\cup F_2)+r(F_1\cap F_2)$: (1) add $(F_1\cap F_2,1)_\leq$ as a restriction, 
(2) add $(F_1\cup F_2,2)_\leq$ as a restriction.

Note that we cannot add inequalities when they contradict that $(A_1\cup A_2,2)_\leq$ appeared in Lemma \ref{lem:2facets} is a facet-defining inequality.
When $F_1$ or $F_2$ is such a flat, the additional inequality should be of type (1).

\begin{lem}\label{lem:forcing1}
Consider a connected simple matroid $M$ of rank 3 as an original matroid.
Assume that ${\cal B}(M)$ has an included matroid base system ${\cal B}(M')$ such that $A$ is a flat of rank 1.
Let $C\subseteq E$ be a (non-empty) connected vertices of $g(A,B)$ where $B$ is a non-empty set disjoint from $A$.
Then $A\cup C$ has rank 2 on $M'$.
\end{lem}

\begin{proof}
We use an induction on the size of connected vertices $C$ of $g(A,B)$.

The case of $|C|=1$ follows from that $A$ is a flat on $M'$.
Consider the case of $C=\{x,y\}$ where $\{x,y\}$ is an edge of $g(A, B)$.
By the definition of $g(A, B)$, there exists $z\in A$ such that $\{x,y,z\}$ has rank 2 on $M$.
Since $A$ is a flat of rank 1  on $M'$, $A\cup \{x,y\}$ has rank 2 on $M'$ by submodularity.

Next we show the inductive step of the proof. Assume that $C$ is connected and  $C\cup x$ is connected on the graph $g(A, B)$. By the induction hypothesis, $A\cup C$ has rank 2  on $M'$. Let $\{x,y\}$ be an edge of $g(A,B)$ with $y\in C$. By the induction hypothesis, $A\cup\{x,y\}$ has rank 2 on $M'$.
Since $A$ is a flat of rank 1 on $M'$, $(A\cup C)\cap (A\cup\{x,y\})=A\cup y$ has rank 2 on $M'$.
By submodularity $r'(A \cup C)+r'(A\cup \{x,y\})\geq r'(A\cup y)+r'(A\cup C\cup x)$ on $M'$, $A\cup C \cup x$ has rank 2 on $M'$.
\end{proof}

\begin{lem}\label{lem:flat2}
Consider a connected simple matroid $M$ of rank 3 as an original matroid.
Assume that ${\cal B}(M)$ has an included matroid base system ${\cal B}(M')$ that has a flat $A$ of rank 2.
Then every connected component of $g(B,A)$ has rank 1 on $M'$ where $B$ is a non-empty set  disjoint from $A$.
\end{lem}

\begin{proof}
When $\{x\}$ is contained in no edges in $g(B,A)$, the statement holds since $M'$ is connected and therefore loopless.
Let $\{x,y\}$ be an edge of $g(B, A)$.
By the definition of $g(B, A)$, there exists $z\in B$ such that $\{x,y,z\}$ has rank 2 on the included matroid $M'$. Let $r'$ be the rank function of the included matroid $M'$. Note that $r'(X)\leq r(X)$ for any $X\subseteq E$ by Lemma \ref{lem:r1r2}.
Since $A$ is a flat of rank 2 on $M'$, we have $r'(A\cup z)=3$.
Therefore, by submodularity $r'(\{x,y,z\})+r'(A)\geq r'(A\cup z)+r'(\{x,y\})$, we have $r'(\{x,y\})=1$.
Since the parallel relation on a matroid is an equivalence relation, the connected component including $\{x,y\}$ has rank 1 on $M'$.
\end{proof}

\begin{exa}
Consider a matroid $M_1$ shown in Example \ref{exa:seven}.
$g(\{b,d\},\{b,d\}^c)$ has edges $\{a,c\}$ and $\{a,e\}$.
Therefore, its connected components are $\{a,c,e\},\{g\},$ and $\{f\}$.
When there exists an included matroid base system with a facet-defining inequality $(\{b,d\},1)_\leq$, 
by Lemma \ref{lem:forcing1}, $r'(\{b,d,a,c,e\})=2$.
By Lemma \ref{lem:flat2}, $r'(\{c,e\})=1$.
In fact, ${\cal B}(M_1)$ has an included matroid base system ${\cal B}(M_2)$ shown in Example \ref{exa:seven}.
\end{exa}

\subsection{Decomposition of a matroid base system of rank 3}

In this subsection, we consider decomposition problem of a connected simple matroid base system of rank 3. A facet of a matroid base system is also a matroid base system.
Note that the matroid base system ${\cal B}(M)\cap (A,r(A))_=$ defined by a facet-defining inequality $(A,r(A))_\leq$ of rank 1 or 2 has also rank 3.
We call the matroid base system defined as a facet the {\it matroid base system on the facet}.
By Theorem \ref{thm:facetbase}, the matroid base system defined as a facet has two connected components:
 one connected component is a matroid base system of rank 1, the other connected component is a matroid base system of rank 2.

Consider a decomposable connected matroid base system ${\cal B}(M)$ of rank 3.
${\cal B}(M)$ has an included matroid base system ${\cal B}(M')$ with a rank function $r'$ in the decomposition. For a non-original facet-defining inequality $(A,r'(A))_\leq$ of ${\cal B}(M')$,
$M$ has another included matroid base system that has the facet defined by the inverse inequality $(A,r'(A))_\geq$ by the definition of the decomposition.
We call such an included matroid base system an {\it included matroid base system on the other side of the facet}.

For two included matroid base systems ${\cal B}(M_1)$ and ${\cal B}(M_2)$ that have a common facet,
the intersection ${\cal B}(M_1) \cap {\cal B}(M_2)$ is the matroid base system on the common facet.
When a matroid $M$ and an included matroid base system ${\cal B}(M')$ of $M$ are given, 
the matroid base system ${\cal B}(M')\cap (A,r'(A))_=$ defined as a facet narrows down possible candidates for included matroid base systems on the other side of the facet since any included matroid base system on the other side must have ${\cal B}(M')\cap (A,r'(A))_=$ as a facet.
Note that the rank of a non-original facet-defining flat is 1 or 2 on $M'$.

\begin{exa}\label{exa:typed}
We consider the matroid shown in Example \ref{exa:seven} again.
This matroid base system ${\cal B}(M)$ is decomposed into four included matroid base systems $M_1, M_2, M_3,$ and $M_4$ as follows. 
Since ${\cal B}(M_1)$ has three non-original facet-defining flats $(\{a,b,c,d,e\},2)_\leq$, $(\{b,d\},1)_\leq$ and $(\{c,e\},1)_\leq$, we try to find an included matroid base system on the other side of each of the three facets. 
The matroid $M_{12}$ on $(\{a,b,c,d,e\},2)_\leq$ is obtained from imposing $(\{f,g\},1)_\leq$ on $M_1$
since $(E,3)_=\cap (\{a,b,c,d,e\},2)_= =(E,3)_=\cap(\{a,b,c,d,e\},2)_\leq \cap (\{f,g\},1)_\leq$.
$${\cal B}(M_{12})=(E,3)_=\cap (\{a,b,c,d,e\},2)_\leq\cap (\{b,d\},1)_\leq\cap (\{c,e\},1)_\leq\cap (\{f,g\},1)_\leq.$$

The included matroid base system ${\cal B}(M_2)$ on the other side of the facet ${\cal B}(M_{12})$ should satisfy the following conditions.
\begin{itemize}
\item ${\cal B}(M_2)$ is included in ${\cal B}(M)$,
\item $(\{f,g\},1)_\leq$ is a non-original facet-defining inequality of ${\cal B}(M_2)$,
\item ${\cal B}(M_{12})$ is the matroid base system on $(\{f,g\},1)_=$ as a facet of ${\cal B}(M_2)$.
\end{itemize}
Note that ${\cal B}(M_{12})={\cal B}(M_1)\cap{\cal B}(M_2)$. So we have
$${\cal B}(M_2)=(E,3)_=\cap (\{a,b,c\},2)_\leq\cap (\{a,d,e\},2)_\leq\cap (\{c,e,f,g\},2)_\leq\cap (\{f,g\},1)_\leq.$$

Similarly, we can find the included matroid base systems in the decomposition other than ${\cal B}(M_1)$ and ${\cal B}(M_2)$.
$${\cal B}(M_3)=(E,3)_=\cap (\{a,b,d,f,g\},2)_\leq\cap (\{c,e,g\},2)_\leq\cap (\{a,b,d\},1)_\leq,$$
$${\cal B}(M_4)=(E,3)_=\cap (\{a,c,e,f,g\},2)_\leq\cap (\{b,d,f\},2)_\leq\cap (\{a,c,e\},1)_\leq.$$

\begin{figure}
\begin{center}
\includegraphics{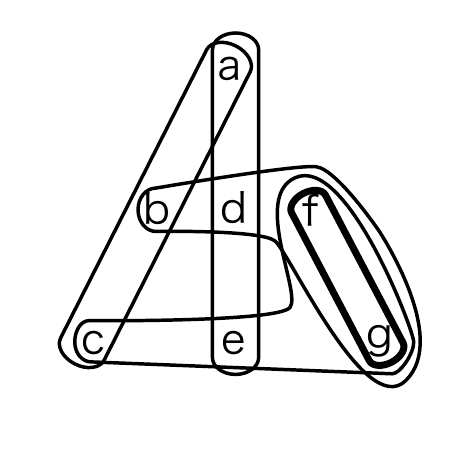}
\includegraphics{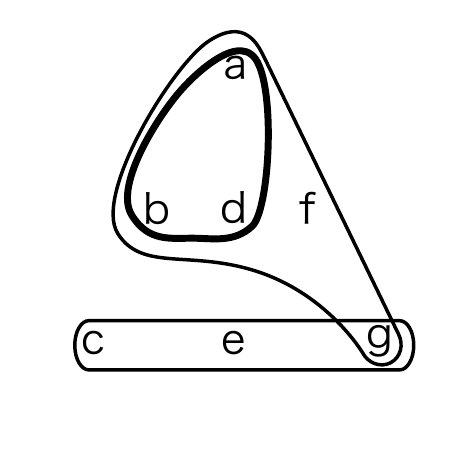}
\includegraphics{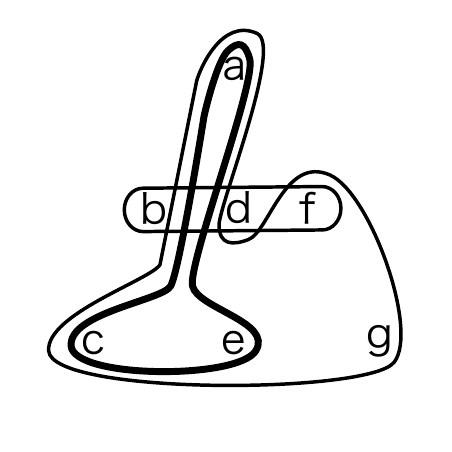}
\caption{Included matroids $M_2$, $M_3$ and $M_4$}\label{fig:matroids}
\end{center}
\end{figure}

Among these four matroids, the following are the pairs of matroids whose base systems  have a common facet.
$$\{M_1,M_2\},\{M_1,M_3\},\{M_1,M_4\},\{M_2,M_3\},\{M_2,M_4\}.$$

Since ${\cal B}(M_3)\cap {\cal B}(M_4)=(\{a\},0)_=\cap (\{b,d\},1)_=\cap (\{c,e\},1)_=\cap (\{f,g\},1)_=$, ${\cal B}(M_3)$ and ${\cal B}(M_4)$ do not have any common facet.
\end{exa}

\begin{lem}\label{lem:3parts}
Consider a connected simple matroid $M$ of rank 3 as an original matroid that has a 3-partition $\{A_1,A_2,A_3\}$.
Assume that an included matroid base system ${\cal B}(M_1)$ has non-original facet-defining inequalities $(A_1,1)_\leq$ and $(A_1\cup A_2,2)_\leq$.
Assume that there exists a decomposition using ${\cal B}(M_1)$.
Then there exists an included matroid base system which has non-original facet-defining inequalities $(A_3,1)_\leq$ and $(A_3\cup A_1,2)_\leq$.
Moreover, there exists another included matroid base system which has non-original facet-defining inequalities $(A_2,1)_\leq$ and $(A_2\cup A_3,2)_\leq$.
\end{lem}

\begin{proof}
Since ${\cal B}(M)$ has a decomposition using ${\cal B}(M_1)$, there exists an included matroid base system on the other side of $(A_1,1)_=$,
 and there exists an included matroid base system on the other side of $(A_1\cup A_2,2)_\leq$.

By considering the matroid base system ${\cal B}(M_1)\cap (A_1\cup A_2,2)_=$ as a facet of ${\cal B}(M)$,
the included matroid base system ${\cal B}(M_2)$ on the other side of $(A_1\cup A_2,2)_\leq$ has non-original facet-defining inequality $(A_3,1)_\leq$. Because $(A_1,1)_=\cap (A_2,1)_=\cap (A_3,1)_=$ is a ridge of the facet ${\cal B}(M_1)\cap (A_1\cup A_2,2)_=$, another facet-defining equality should contain this ridge. Therefore, the included matroid base system ${\cal B}(M_2)$ has non-original facet-defining inequality $(A_3\cup A_1,2)_\leq$.

By considering the matroid base system ${\cal B}(M_1)\cap (A_1,1)_=$ on the facet, 
the included matroid base system on the other side of $(A_1,1)_\leq$ has non-original facet-defining inequalities $(A_2\cup A_3,2)_\leq$ and $(A_2,1)_\leq$.
\end{proof}




Consider small matroid base systems of rank 3 which are neither binary nor 2-decomposable by our computation.
There exists no such matroid of size 6. 
There exist two such matroids of size 7. 
There exist five such matroids of size 8. 

\section{Main results about matroid base systems}\label{sec:further}

This section is the main part of this paper and consists of two subsections.
In Section \ref{subsec:classification}, we classify the matroid base systems according to their decomposability. In Section \ref{subsec:counter}, we give a counterexample to the conjecture proposed by Lucas \cite{Lucas75}.

\subsection{A classification of matroid base systems}\label{subsec:classification}

In this subsection, we classify the matroid base systems according to their decomposability.

Lucas proved the following theorem.

\begin{thm} (Lucas \label{thm:lucas}\cite{Lucas75})
There exists no connected binary matroid whose base system includes that of another connected matroid on the same ground set.
\end{thm}

In other words, a connected binary matroid is minimal in all the connected matroids on the same ground set with respect to weak-map order of the same rank.

We have the next corollary from Lemma \ref{lem:divweak}.

\begin{cor}\label{cor:minimal}
No connected matroid base system which is minimal with respect to weak-map order in all the connected matroids is decomposable.
\end{cor}



We classify matroid base systems into five types (a), (b), (c), (d), and (e) with respect to their decomposability and weak-map order.

(a) Binary matroids.\\
A binary matroid is minimal in the connected matroids with respect to inclusion by Theorem \ref{thm:lucas}.
 Therefore its base system is not decomposable by Corollary \ref{cor:minimal}.

(b) Non-binary but minimal matroids in the connected matroids with respect to inclusion.\\



A connected matroid base system of rank 4 which does not include any included matroid base system is known (Lucas \cite{Lucas75}).
We give a connected matroid of rank 3 which does not include any included matroid base system.

\begin{exa}\label{exa:minimal}
Let $E=\{a,b,c,d,e,f,g,h\}$.
The following connected simple matroid base system is an example of a matroid whose base system does not include any included matroid base system. 

$${\cal B}(M)=(E,3)_=\cap (afd,2)_\leq \cap (ebh,2)_\leq \cap (abg,2)_\leq\cap (efc,2)_\leq \cap (egd,2)_\leq\cap (ach,2)_\leq \cap(bcd,2)_\leq \cap (fgh,2)_\leq.$$

These flats define a matroid because any two flats of rank 2 intersect in at most one element.
This matroid is shown in Figure \ref{fig:minimal} where each line represents a flat of rank 2.
\begin{figure}
\begin{center}
\includegraphics{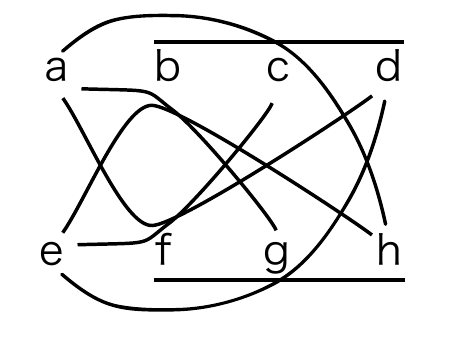}
\caption{A matroid $M$ which does not include any included matroid}\label{fig:minimal}
\end{center}
\end{figure}

Moreover, it is not binary since $M/\{a\}\backslash\{d,g,h\}=U_{2,4}$.
We prove that it does not include any included matroid base system in Theorem \ref{thm:minimal}.
\end{exa}

\begin{lem}\label{lem:minimal}
The matroid base system shown in Example \ref{exa:minimal} has no 3-partition.
\end{lem}

\begin{proof}
By Lemma \ref{lem:hatonosu}, no 3-partition is of size $(3,3,2)$ because $e(3)+e(3)+e(2)=7 < f(3)\times 8=8$. 
Therefore the size of the 3-partition $\{A_1,A_2,A_3\}$ is $ (|A_1|,|A_2|,|A_3|)=(4,2,2)$. 
Since any original facet-defining flat of rank 2 cannot intersect all of $A_1, A_2$ and $A_3$, any original facet-defining  flat of rank 2 intersects one of $A_1, A_2$ or $A_3$ in at least two elements. Note that
any original facet-defining flat of rank 2 has size at least 3.
Since $M$ is simple, any two original facet-defining flats of rank 2 intersect in at most one element by Lemma \ref{lem:ranktwoes}.
Since $e(4)+e(2)+e(2)=8$, each edge of the graph in Figure \ref{fig:622} is included in exactly one original facet-defining flat of rank 2 by a similar argument to the proof of Lemma \ref{lem:hatonosu}.

\begin{figure}
\begin{center}
\includegraphics{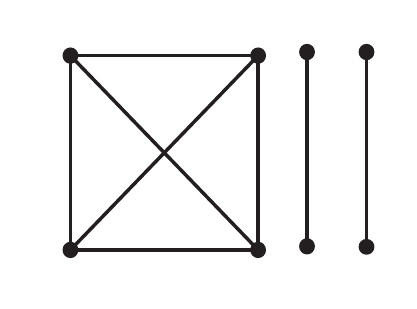}
\caption{Complete graphs of $A_1, A_2$ and $A_3$}\label{fig:622}
\end{center}
\end{figure}

By Lemma \ref{lem:graphdis}, $g(A_2, A_1)$ and $g(A_3, A_1)$ have no common edge.
Since the number of the edges in the complete graph on $A_1$ is 6,
the maximum of $g(A_2, A_1)+g(A_3, A_1)$ is 6.
Since the number of the facet-defining flats of rank 2 is 8, this implies that $A_2\cup A_3$ includes at least two original facet-defining flats of rank 2, which contradicts that two facet-defining flats intersect in at most one element by Lemma \ref{lem:ranktwoes}.
\end{proof}

\begin{thm}\label{thm:minimal}
The matroid base system shown in Example \ref{exa:minimal} does not include any included matroid base system.
\end{thm}

\begin{proof}
Suppose that the matroid base system ${\cal B}(M)$ has an included matroid base system ${\cal B}(M')$.
Note that this matroid base system ${\cal B}(M)$ has 8 facet-defining flats of rank 2 and is not 2-decomposable. By Lemma \ref{lem:2facets}, ${\cal B}(M')$ has a 3-partition in $M$, which contradicts Lemma \ref{lem:minimal}.
\end{proof}






(c) Non-binary and non-minimal but indecomposable matroids.\\
We give an example of type (c) of size 8 in Example \ref{exa:nonminimal}.

\begin{exa}\label{exa:nonminimal}
The next example is an indecomposable matroid base system which includes an included matroid base system.

Let $E=\{a,b,c,d,e,f,g,h,i\}$.
$${\cal B}(M)=(E,3)_=\cap (bdfh,2)_\leq\cap(abc,2)_\leq\cap (ade,2)_\leq\cap (afg,2)_\leq \cap (ahi,2)_\leq \cap (bgi,2)_\leq \cap (cdi,2)_\leq \cap (cef,2)_\leq \cap (egh,2)_\leq.$$

The matroid $M$ is depicted in the left figure of Figure \ref{fig:nonminimal} where each line represents a facet-defining flat of rank 2.

\begin{figure}[ht]
\begin{center}
\includegraphics{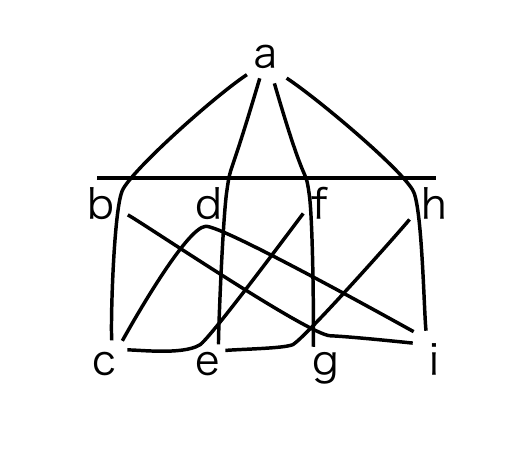}\quad \includegraphics{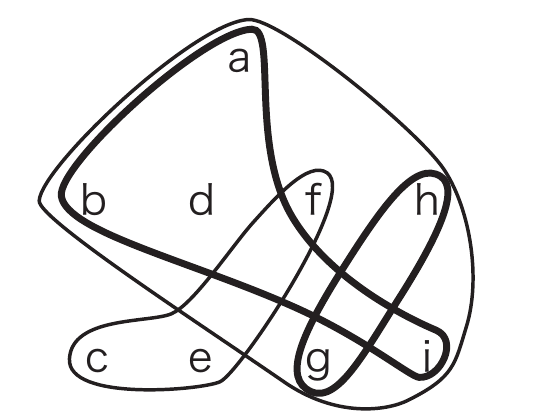}
\caption{Indecomposable original matroid $M$ and its included matroid $M_1$}\label{fig:nonminimal}
\end{center}
\end{figure}

This matroid base system includes the following included matroid base system shown in Figure \ref{fig:nonminimal}.

$${\cal B}(M_1)=(E,3)_=\cap (\{a,b,d,i\},1)_\leq\cap(\{g,h\},1)_\leq\cap (\{a,b,d,g,f,h,i\},2)_\leq\cap (\{c,e,f\},2)_\leq.$$

We can check easily that ${\cal B}(M_1)\subseteq {\cal B}(M)$ by Lemma \ref{lem:facetinc}.
However, by Theorem \ref{thm:nondiv} below,  ${\cal B}(M)$ is not decomposable.
\end{exa}

\begin{thm}\label{thm:nondiv}
The matroid base system shown in Example \ref{exa:nonminimal} is not decomposable.
\end{thm}

\begin{proof}
The matroid base system ${\cal B}(M)$ has one original facet-defining flat of size 4, and eight original facet-defining flats of size 3.
By Theorem \ref{thm:2div}, the matroid base system ${\cal B}(M)$ is not 2-decomposable.
By Lemma \ref{lem:2facets}, an included matroid base system ${\cal B}(M')$ has a 3-partition in $M$.
Since $e(3)+e(3)+e(3)=9$ and $f(4)+8\times f(3)=10$,
 $M$ has no 3-partition of size $(3,3,3)$ by Lemma \ref{lem:hatonosu}.
Therefore the size of a 3-partition is $(2,3,4)$ or $(2,2,5)$.
By Lemma \ref{lem:2facets} and Lemma \ref{lem:3parts},
we can assume that one of $A_1, A_2$, and $A_3$ has rank 1 on some included matroid base system.

Since this matroid has symmetry,
we have to consider decomposability only for $ab,ac,bc,be,bg,cg,bd,bf,ce$.
For each case of $X=\{b,e\}, \{c,g\}$, its graph $g(X,X^c)$ is connected.
Therefore it cannot have a 3-partition with connected component $X$ by Lemma \ref{lem:flat2}.
For each case of $X=\{a,b\},\{a,c\},\{b,c\},\{b,g\}$, $g(X,X^c)$ has two connected components and one of them has size 1.
By using Lemmas \ref{lem:forcing1} and \ref{lem:flat2}, we can show that there exists no included matroid base system with such a flat of rank 1.

As for $\{b,d\}$, $g(\{b,d\}, \{a,c,e,f,g,h,i\})$ has two connected components $\{f,h\}$, and $\{a,c,e,i,g\}$ since its edges are $f-h$ and $e-a-c-e-g$. However, since $\{b,d,f,h\}$ is of rank 2 on $M$, no included matroid base system has a 3-partition in $M$ with a flat $\{b,d\}$.
The case of $\{b,f\}$ is similar.

As for $\{c,e\}$, $g(\{c,e\}, \{a,b,d,f,g,h\})$ has three connected components $\{f\}, \{g,h\}$, and $\{a,b,d,i\}$. 
Therefore, on an included matroid $M'$ with a flat $\{c,e\}$, by Lemma \ref{lem:forcing1}, $\{c,e,g,h\}$ and $\{c,e,a,b,d,i\}$ have rank 2. 
$g(\{g,h\}, \{a,b,d,i\})$ is connected. 
Since $E-\{f\}$ cannot have rank 2, $\{a,b,c,d,e,i\}$ is a flat of rank 2 on $M'$.
Therefore, by Lemma \ref{lem:flat2}, $\{a,b,d,i\}$ has rank 1 on $M'$.
Since an included matroid $M'$ is connected, $\{a,b,d,i\}$ is a flat.
Therefore $g(\{a,b,d,i\}, \{g,h,f\})$ is connected, 
the rank of $\{a,b,d,i,g,h,f\}$ is 2, which contradicts the connectivity of the included matroid $M'$.
\end{proof}

(d) Non-binary and non-2-decomposable but decomposable matroids.\\

The matroid in Example \ref{exa:typed} is decomposable but not 2-decomposable.

Billera et al. \cite{Billera09} gave an example of a matroid decomposition consisting of three matroid base systems. However, it is also 2-decomposable. 

(e) Non-binary and 2-decomposable matroids.\\
You can easily find a lot of examples of this type, for example, uniform matroid $U_{2,4}$.




\subsection{A counterexample to Lucas's conjecture}\label{subsec:counter}

Lucas proposed the following conjecture.

\begin{con}\label{con:lucas}\cite{Lucas75}
  Assume that a matroid $M_1$ covers a matroid $M_2$ with respect to
  the weak-map order among matroid base polytopes of rank $r$ on $E$,
  that is, ${\cal B}(M_2)\subseteq {\cal B}(M_1)$ and there exists no matroid $M_3$ such that ${\cal B}(M_2)\subsetneq {\cal B}(M_3)\subsetneq {\cal B}(M_1)$.
  Moreover assume that the dimension of ${\cal B}(M_2)$ is less than that of ${\cal B}(M_1)$. Then ${\cal B}(M_2)$ is a facet of ${\cal B}(M_1)$.
\end{con} 

We give a counterexample to this conjecture.

\begin{exa}\label{exa:lucascon}
The pair $(M_1,M_2)$ of matroids below is a counterexample to Conjecture \ref{con:lucas} as shown in Theorem \ref{thm:lucascon}.

Let ${\cal B}(M_1)$ be the matroid base system such that the facet-defining flats of rank 2 consist of
$$abk,bce,cdi,adf,bdh,acj,efi,fgc,ghj,hea,egd,fhk,ijb,jke,kig.$$
Note that any two facet-defining flats of rank 2 intersect in at most 1 element.
So these flats define a matroid.
The facet-defining flats of this matroid are illustrated in the left of Figure \ref{fig:lucas}. In this figure, for example, facet-defining flat $\{a,b,k\}$ is represented as edge $\{a,b\}$ with label $k$.
Thus, the 12 flats correspond to the 12 edges in the graph.
Let ${\cal B}(M_2)$ be the matroid base system defined by 
$${\cal B}(M_2)=(E,3)_=\cap (\{a,b,c,d\},1)_\leq\cap (\{e,f,g,h\},1)_\leq\cap (\{e,f,g,h,i,j,k\},2)_\leq \cap (\{i,j\},1)_\leq,$$
as illustrated in the right figure of Figure \ref{fig:lucas}. Note that $M_2$ has two connected components $\{a,b,c,d\}$ and $\{e,f,g,h,i,j,k\}$. Note that ${\cal B}(M_2)\subseteq {\cal B}(M_1)$.
\begin{figure}
\begin{center}
\includegraphics{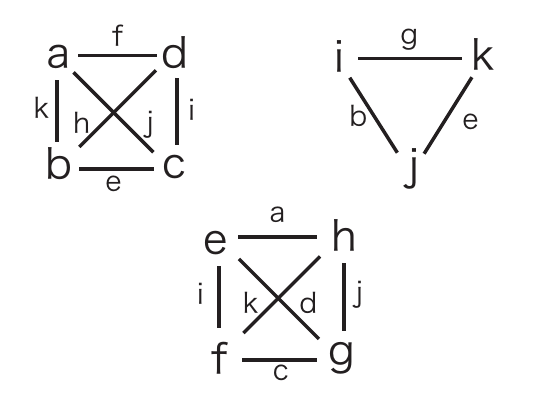}\quad \quad \quad \quad \quad \includegraphics{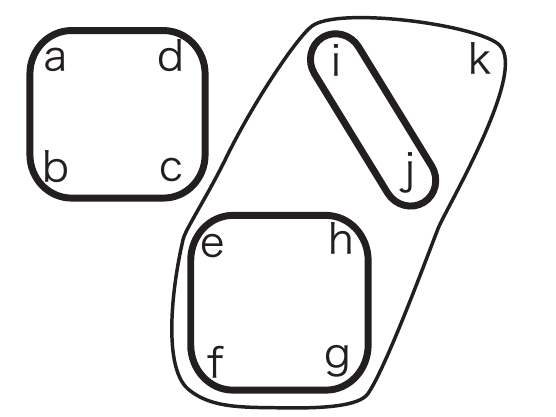}
\caption{A counterexample pair to Lucas's conjecture: $M_1$ and $M_2$}\label{fig:lucas}
\end{center}
\end{figure}
\end{exa}

\begin{lem}\label{lem:lucascon1}
${\cal B}(M_1)$ has no included matroid base system that has $(\{a,b,c,d\},1)_\leq$ as a non-original facet-defining inequality.
\end{lem}

\begin{proof}
Suppose that $M_1$ has an included matroid base system ${\cal B}(M')$ with $(\{a,b,c,d\},1)_\leq$ as a non-original facet-defining inequality.
Since $g(\{a,b,c,d\},\{e,f,g,h,i,j,k\})$ has edges $\{e,h\},\{e,g\},\{f,g\},\{i,j\}$, it  has connected components $\{e,f,g,h\}$ and $\{i,j\}$.
Therefore the included matroid base system ${\cal B}(M')$ satisfies 
$(\{a,b,c,d,e,f,g,h\},2)_\leq$ and $(\{a,b,c,d,i,j\},2)_\leq$ by Lemma \ref{lem:forcing1}.

Consider the case where $\{a,b,c,d,e,f,g,h\}$ is a flat on the included matroid $M'$.
$\{e,f,g,h\}$ is a connected component of $g(\{i,j,k\},\{a,b,c,d,e,f,g,h\})$ since
this graph has edges $\{e,f\},\{f,h\},\{h,g\}$.
Therefore the included matroid base system satisfies $(\{e,f,g,h\},1)_\leq$ by Lemma \ref{lem:flat2}.
When $\{e,f,g,h\}$ is a flat of rank 1 on $M'$, the included matroid base system ${\cal B}(M')$ satisfies $(\{e,f,g,h,i,j,k\},2)_\leq$ by Lemma \ref{lem:forcing1}, 
which contradicts the connectivity of the included matroid ${\cal B}(M')$.
When $\{e,f,g,h\}$ is not a flat,
the included matroid base system ${\cal B}(M')$ satisfies $(\{e,f,g,h,i,j\},1)_\leq$ or $(\{e,f,g,h,k\},1)_\leq$, which contradicts the connectivity of the included matroid $M'$.

Consider the case where $\{a,b,c,d,e,f,g,h\}$ is not a flat of the included matroid $M'$. Then there exists a flat of rank 2 including $\{a,b,c,d,e,f,g,h\}$. Consider the case where the included matroid  ${\cal B}(M')$ has a flat $\{a,b,c,d,e,f,g,h,k\}$ of rank 2. We have $r'(\{g,k\})=1$ since $r'(\{i,k,g\})\leq 2$ and Lemma \ref{lem:ranktwoes} where $r'$ is the rank function of $M'$. Similarly $r'(\{k,e\})=1$. Eventually, the included matroid ${\cal B}(M')$ becomes non-connected to satisfy submodularity, a contradiction. So this case cannot arise. The case where $\{a,b,c,d,e,f,g,h,i\}$ is a flat can be similarly proved, and so on.
\end{proof}

\begin{lem}\label{lem:lucascon2}
${\cal B}(M_1)$ has no included matroid base system that has $(\{e,f,g,h,i,j,k\},2)_\leq$ as a non-original facet-defining inequality.
\end{lem}

\begin{proof}
Assume that $M_1$ has an included matroid base system ${\cal B}(M')$ with $(\{e,f,g,h,i,j,k\},2)_\leq$ as a non-original facet-defining inequality.
Since $g(\{a,b,c,d\},\{e,f,g,h,i,j,k\})$ has two connected components $\{e,f,g,h\}$ and $\{i,j\}$,
the included matroid base system ${\cal B}(M')$ satisfies $(\{e,f,g,h\},1)_\leq$ and $(\{i,j\},1)_\leq$ by Lemma \ref{lem:flat2}.
Since $g(\{e,f,g,h\},\{a,b,c,d\})$ is connected, the included matroid base system ${\cal B}(M')$ satisfies $(\{a,b,c,d,e,f,g,h\},2)_\leq$  by Lemma \ref{lem:forcing1}.
When $(\{a,b,c,d,e,f,g,h\},2)_\leq$ is a facet-defining flat, the included matroid base system ${\cal B}(M')$ satisfies $(\{a,b,c,d\},1)_\leq$ by Lemma \ref{lem:flat2},
 which contradicts the connectivity of the included matroid $M'$.
When $(\{a,b,c,d,e,f,g,h\},2)_\leq$ is not a facet-defining flat of ${\cal B}(M')$, the included matroid base system ${\cal B}(M')$ satisfies $(\{a,b,c,d,e,f,g,h,k\},2)_\leq$ 
or $(\{a,b,c,d,e,f,g,h,i,j\},2)_\leq$, which contradicts the connectivity of the included matroid $M'$.
\end{proof}

\begin{thm}\label{thm:lucascon}
The pair of $M_1$ and $M_2$ is a counterexample to Conjecture \ref{con:lucas}.
\end{thm}

\begin{proof}
Checking ${\cal B}(M_2)\subseteq {\cal B}(M_1)$ is straightforward by Lemma \ref{lem:facetinc}. Assume that Conjecture \ref{con:lucas} holds.
Then there exists a matroid base system $M_3$ such that ${\cal B}(M_2)\subseteq {\cal B}(M_3)\subsetneq {\cal B}(M_1)$ and ${\cal B}(M_3)$ is a facet of ${\cal B}(M_1)$.
Since ${\cal B}(M_2)\subseteq (\{a,b,c,d\},1)_=$, ${\cal B}(M_3)$ has a facet-defining inequality $(\{a,b,c,d\},1)_\leq$ or $(\{e,f,g,h,i,j,k\},2)_\leq$.
By Lemma \ref{lem:lucascon1} and Lemma \ref{lem:lucascon2}, 
${\cal B}(M_1)$ has no included matroid with ${\cal B}(M_2)$ as a facet.
\end{proof}

\bibliographystyle{springer}
\bibliography{all}

\end{document}